\documentclass[letter, 11pt]{article}
\usepackage[english]{babel}
\usepackage[utf8]{inputenc}

\usepackage[margin = 1.2 in, top = 1 in, bottom = 1.2 in]{geometry}
\makeatletter
\g@addto@macro\normalsize{%
  \setlength\abovedisplayskip{7pt}
  \setlength\belowdisplayskip{7pt}
  \setlength\abovedisplayshortskip{7pt}
  \setlength\belowdisplayshortskip{7pt}
}
\makeatother

\usepackage{graphicx}
\usepackage{cancel}

\usepackage{enumerate}
\usepackage{scalerel,stackengine}
\usepackage{graphics}
\usepackage{fancyhdr}
\usepackage{cancel}
\usepackage{enumerate}
\usepackage{amssymb}
\usepackage{amsmath}
\usepackage{hyperref}
\usepackage{mdwlist}
\usepackage{etoolbox}
\usepackage{latexsym}
\usepackage{amsthm}
\usepackage{multicol}
\usepackage{makeidx}
\usepackage{bm}
\usepackage{dsfont}
\usepackage{graphicx}
\usepackage{wrapfig}
\usepackage{multicol}
\usepackage{makeidx}
\usepackage{bm}
\usepackage{dsfont}
\usepackage{mdframed,color}
\usepackage{amsmath,amssymb}
\usepackage[dvipsnames]{xcolor}
\usepackage{mathtools}
\usepackage{comment}

\usepackage{soul}

\usepackage[capitalize]{cleveref}
\usepackage{bbm, dsfont}
\usepackage{tikz}
\usetikzlibrary{matrix}
\usetikzlibrary{positioning}
\usetikzlibrary{arrows.meta}
\usepackage{titlesec}
\titleformat{\section}
  {\large\center\bfseries}
  {\thesection.}{.7em}{}
\titlespacing*{\section}{0pt}{3.5ex plus 0ex minus 0ex}{1.5ex plus 0ex}
\titleformat{\subsection}
  {\center\bfseries}
  {\thesubsection.}{.7em}{}
\titlespacing*{\subsection}{0pt}{3.5ex plus 0ex minus 0ex}{1.5ex plus 0ex}
\titleformat{\subsubsection}
  {\center\bfseries}
  {\thesubsubsection.}{.7em}{}
\titlespacing*{\subsubsection}{0pt}{3.5ex plus 0ex minus 0ex}{1.5ex plus 0ex}

\addto\captionsenglish{}

\usepackage{titling}
\newcommand{\Erdos}{Erd\H{o}s}
\newcommand{\Folner}{F\o{}lner}

\newtheorem{theo}{Theorem}[section]
	
\newtheorem{cor}[theo]{Corollary}

\newtheorem{question}[theo]{Question}
\newtheorem{lemma}[theo]{Lemma}
\newtheorem{prop}[theo]{Proposition}
\newtheorem*{prop*}{Proposition}
\newtheorem{thmx}{Theorem}
\newtheorem{corox}[thmx]{Corollary}

\newtheoremstyle{definition}{2mm}{2mm}{}{}{\bfseries}{.}{.5em}{}
\theoremstyle{definition}
\newtheorem{defn}[theo]{Definition}
\newtheorem{remark}[theo]{Remark}
\newtheorem*{remark*}{Remark}

\theoremstyle{plain}
\newtheorem*{namedthm}{\namedthmname} 

\makeatletter
\newenvironment{named}[2]{%
  \def\namedthmname{#1}
  \begin{namedthm}[#2]%
  \def\@currentlabelname{#1}
}{%
  \end{namedthm}%
}
\makeatother

\newcommand{\N}{\mathbb{N}}
\newcommand{\Z}{\mathbb{Z}}
\newcommand{\C}{\mathbb{C}}

\newcommand{\R}{\mathbb{R}}

\newcommand{\cM}{\mathcal{M}}

\DeclareMathOperator\supp{supp}
\DeclareMathOperator\gen{gen}
\newcommand{\T}{\mathbb{T}}
\newcommand{\norm}[1]{\left\lVert#1\right\rVert}

\newcommand*\diff{\mathop{}\!\mathrm{d}}

\providecommand{\norm}[1]{\lVert #1\rVert}

\newcommand{\1}{\mathbbm{1}}
\newcommand{\E}{\operatorname{\mathbb{E}}}

\begin{document}
\author{By~~{\scshape Felipe~Hernández}~~and~~{\scshape Ioannis~Kousek }~~and~~{\scshape Tristán~Radić}}
\date{\small \today}
\title{{\textbf{On density analogs of Hindman's finite sums theorem}}}
\maketitle

\begin{abstract}
For any set $A$ of natural numbers with positive upper Banach density, we show the existence of an infinite set $B$ and sequences $(t_k)_{k\in \N}, (s_k)_{k\in \N}$ of natural numbers such that $\left\{ \sum_{n \in F}n : F \subset B + s_k, 1 \leq |F|  \leq k \right\}\subset A-t_k$, for every $k\in \N$. This strengthens the density finite sums theorem of Kra, Moreira, Richter, and Robertson. We further show, given such a set $A$, the existence of an infinite set $B$ and a sequence 
$(t_k)_{k\in \N}$ of natural numbers such that $\left\{ \sum_{n \in F}n : F \subset B, |F|  = k \right\}\subset A-t_k$, for every $k\in \N$. As a corollary, we obtain a sequence $(B_n)_{n\in \N}$ of infinite sets of natural numbers such that $B_1+\cdots +B_k \subset A$, for every $k\in \N$. We also establish the optimality of our main theorems by providing 
counterexamples to potential further generalizations, and thereby addressing 
questions of the aforementioned authors in the context of density analogs to 
Hindman’s finite sums theorem.
\end{abstract}

\section{Introduction}

\subsection{Main results}
In 1974, Hindman \cite{Hindman74} proved the celebrated finite sums theorem, showing that for any finite partition of the natural numbers, $\N = C_1 \cup C_2 \cup \cdots \cup C_k$, there exists an infinite set $B \subset \N$ such that 
\begin{equation}\label{Hindman's theorem sumsets}
\Big\{ \sum_{n \in F}n \colon F \subset B, \ 1\leq |F| < \infty \Big\} \subset C_i, \quad \text{ for some } i=1, \ldots, k.
\end{equation}
Erd\H{o}s \cite{Erdos77, erdHos1980survey,erdos2006problems} was interested in 
extending Hindman's result to a density analog, namely by finding infinite 
sumset patterns similar to those in \eqref{Hindman's theorem sumsets}, in any set $A\subset \N$ whose upper (Banach) density,
\begin{equation*}
    \diff^{*}(A)=\limsup_{N-M\to \infty}\frac{|A\cap \{M,M+1,\ldots,N-1\}|}{N-M},
\end{equation*}
is positive. Due to parity obstructions -- consider, for instance, the odd numbers -- shifting the finite sumsets in \eqref{Hindman's theorem sumsets} is necessary. An example of Straus \cite[Theorem $11.6$]{Hindman_ultrafilters} showing that a single shift of these patterns cannot always be found in sets of positive density, rendered unclear what could potentially be a density analog of Hindman's theorem.

A series of breakthroughs in the last decade \cite{diNasso_Golbring_Jin_Leth_Lupini_Mahlburg2015sumset,host2019short,Kra_Moreira_Richter_Robertson:2022,Kra_Moreira_Richter_Robertson:2023,kmrr25, Moreira_Richter_Robertson19} revealed that infinite sumset patterns of increasing complexities can be found in sets of positive upper Banach density. These results 
verified conjectures predicted by Erd\H{o}s and gave rise to further natural questions and conjectures (most of which were formulated in \cite{Kra_Moreira_Richter_Robertson_problems}). In particular, using structure theory of multiple ergodic averages, Kra, Moreira, Richter and Roberson proved the following truncated density version of Hindman's theorem.
\begin{named}{Theorem KMRR}{{\cite[Theorem 1.1]{kmrr25}}}\label{main theorem kmrr 1}
Let $A\subset \N$ be a set with $\diff^*(A)>0$. Then, for every $k\in \N$, there exists an infinite set $B_k \subset \N$ and an integer $t_k\geq 0$ such that 
\begin{equation*} 
    \bigg\{ \sum_{n\in F} n \colon \ F\subset B_k \text{ with } 1\leq |F|\leq k \bigg\} \subset A-t_k.
\end{equation*}
\end{named}
Again, Straus's example shows that one 
cannot choose the set $B_k$ and the shift $t_k$ independently of $k$.
However, the same authors asked if a sequence of shifts to the dense sets would allow for an unbounded number of summands generated by a fixed infinite set.
\begin{question}[{\cite[Question 2.11]{Kra_Moreira_Richter_Robertson_problems}}]  \label{question kmrr intro}
Let $A\subset \N$ be a set with $\diff^*(A)>0$. Does there exist an infinite set $B \subset \N$ and a sequence of shifts $(t_k)_{k\in \N}\subset \N$ such that
$$ \bigg\{ \sum_{n\in F} n \colon \ F\subset B \text{ with } 1\leq |F|\leq k \bigg\} \subset A-t_k,\ \text{for every}\ k\in \N\ ?$$
\end{question}

We answer Question \ref{question kmrr intro} 
in the negative (see \cref{1.2 is optimal (ii)}) by constructing a set of density 
arbitrarily close to $1$ for which even a weaker version of the conclusion  
fails. However, our main results, Theorems \ref{theo-A} and \ref{theo-B}, extend 
\nameref{main theorem kmrr 1} in two different directions, achieving in each case a certain degree of independence between the infinite set $B$ and the number of summands $k \in \N$. 
\begin{thmx}\label{theo-A}
    Let $A\subset\N$ be a set with $\diff^*(A)>0$ and $\ell\geq 0$ an integer. Then, there exist an infinite set $B_{\ell} \subset \N$ and a sequence $(t_{\ell,k})_{k\in \N}\subset \N$, such that 
        $$ \bigg\{ \sum_{n\in F} n \colon \ F\subset B_{\ell} \text{ with } k\leq |F|\leq k+ \ell \bigg\} \subset A-t_{\ell,k}, \ \text{for every}\ k\in \N.$$
\end{thmx}

In Section \ref{counterexamples section}, we show that 
\cref{theo-A} is optimal in the sense that one cannot choose the set 
$B_{\ell}$ independently of $\ell$. 
We also show that the sequence of shifts $(t_k)_{k \in \N}$ cannot, in general, be bounded. Moreover, by the proof of \cref{theo-A}, we can choose the sequence $(t_{\ell,k})_{k \in \N}$ to be increasing.

\begin{thmx}\label{theo-B}
    Let $A\subset\N$ be a set with $\diff^*(A)>0$. Then, there exist an infinite set $B$ and sequences $ (t_k)_{k\in \N}, 
    (s_k)_{k\in \N} \subset \N$ such that
    $$ \bigg\{ \sum_{n\in F} n \colon \ F\subset B+s_k \text{ with } 1\leq |F|\leq k \bigg\} \subset A-t_k ,\ \text{for every}\ k\in \N.$$
\end{thmx}

Letting $B_k=B+s_k$, for every $k\in \N$, we see that the strength of this result over \nameref{main theorem kmrr 1} is that the sets $(B_k)_{k\in \N}$ are shifts of a single infinite set $B\subset \N$. Our construction in Section \ref{counterexamples section} shows that one cannot in general remove the shift sequences $(t_k)_{k\in \N},(s_k)_{k\in \N}$, as well as the necessity of them being unbounded. The second auxiliary sequence in \cref{theo-B}, as opposed to merely shifting the dense set, is naturally introduced to overcome the further local obstructions that arise in answering \cref{question kmrr intro}. These results are the first to guarantee sumsets 
with arbitrarily large numbers of summands generated by
fixed infinite sets, contained in sets of positive density. Furthermore, we prove a combination of the conclusions of Theorems \ref{theo-A} 
and \ref{theo-B} (see \cref{Mix-of-Theo-A-and-B} for more details).

The main result of \cite{Kra_Moreira_Richter_Robertson:2022} shows that for any 
$k\in \N$, any set of positive density contains a sumset of the form 
$B_1+\cdots + B_k$, where $B_1,\ldots,B_k$ are infinite. As a corollary of our 
main theorems, we strengthen this by removing the dependence of the infinite sets on $k$. This 
positively answers \cite[Question 8.1]{Kra_Moreira_Richter_Robertson:2022} and 
verifies \cite[Conjecture 2.10]{Kra_Moreira_Richter_Robertson_problems}.
\begin{corox}\label{coro-C}
    Let $A\subset\N$ be a set with $\diff^*(A)>0$. Then, there exists an infinite sequence $(B_n)_{n\in \N}$ of infinite sets of natural numbers such that 
    $$B_1 + B_2 + \cdots + B_k\subset A,\ \text{for every}\ k\in \N.$$
\end{corox}
We remark that \cref{theo-A} for $\ell=0$ already implies 
\cref{coro-C} and the details can be found in 
\cref{translation to dynamics and some reductions}. Similarly, \cref{theo-B} 
also implies \cref{coro-C}.

\subsection{Dynamical reformulation and key ingredients} \label{sec outline of main argument}

In this section, we summarize our approach to prove \cref{theo-A} and \cref{theo-B}, highlighting the new tools developed to address them. We make use of notation and terminology that is later defined in \cref{notational conventions} and \cref{Terminology}.

 Inspired by the aforementioned literature \cite{hernandez2025infinite,host2019short,Kra_Moreira_Richter_Robertson:2023,Kra_Moreira_Richter_Robertson:2022,kmrr25}, after using Furstenberg's correspondence principle, we reinterpret the combinatorial statements into an ergodic language. 
Concerning \cref{theo-A}, the dynamical result that implies it is the following. 

\begin{theo}\label{dynamical reform A}
Let $(X,\mu,T)$ be an ergodic measure preserving system with a point $a\in X$ such that $\mu\left(\overline{\{T^na: n\in \N\}}\right)=1$. Let also $\ell \geq 0$ be fixed and $E\subset X$ be an open set with $\mu(E)>0$. Then, there exist an infinite set $B \subset \N$ and sequence $(t_k)_{k\in \N} \subset \N$ such that 
\begin{equation} \label{eq rec for block equation}
    T^{\left(t_k+ b_{1}+\cdots + b_{i}\right)}a \in E,
\end{equation}
for all $k\in \N$, all $k\leq i \leq k+\ell$, and $b_1, b_2, \ldots, b_i $ distinct elements in $B$.
\end{theo}
To prove a similar dynamical statement, the authors of \cite{kmrr25} introduce a family of measures on finite product spaces, which they call \emph{progressive measures} (see \cref{definition left-right progressive measure}). The positivity of finite cartesian products of open sets with respect to such measures is used to recover \nameref{main theorem kmrr 1}. However, since we are looking for sumset patterns with an unbounded number of summands, we need to deal with infinite cartesian products of open sets and in general, these sets have zero measure with respect to any useful measure. 

We then first adapt the definition of 
progressive measures to measures on $X^{\N_0}$ and build one, denoted by $\sigma \in \cM(X^{\N_0})$, by lifting a natural measure $\xi$ on $Z^{\N}_{\infty}$, where $(Z_{\infty},m,T)$ is the infinite-step pronilfactor. This is a natural approach, and it turns out that the marginals
of this measure $\sigma$ in the first $k+1$ coordinates do not necessarily coincide with the measures constructed in \cite{kmrr25} by lifting measures from the $k$-step pronilfactor $Z_{k-1}$ (see \cref{appendix progresive measures in Zk} for relevant examples). 
One can actually recover the proof of \nameref{main theorem kmrr 1} by considering $\sigma$ lifted from $Z_{\infty}$, but for this result the original measures of \cite{kmrr25} give an optimal control. 

Another obstacle is that we cannot obtain our main results from statements about convergence of orbits in dynamical systems. In previous works, the introduction of Erd\H{o}s progressions (see \cite[Definition 2.1]{kmrr25}) was critical for obtaining sumset-type configurations; see, for instance, \cite{ackelsberg2025polynomial_patterns_rationals,ackelsberg_jamnesham2025equidistribution,charamaras_kousek_mountakis_radic2025BBingroups,Charamaras_Mountakis_2025,kousek2025asymmetric,kousek_radic2025BB,Kra_Moreira_Richter_Robertson:2022,kmrr25,Moreira_Richter_Robertson19}. In contrast, the use of Erd\H{o}s progressions could only lead to weaker versions of Theorems \ref{theo-A} and \ref{theo-B}, in particular versions with nested infinite sets, as we explain in \cref{example appendix Erdos progressions}. Instead, expanding on the ideas already present in \cite{hernandez2025infinite}, we construct the sumsets inductively, using properties of the measure $\sigma$. 

For the following discussion and to illustrate the properties of $\sigma$, fix $k= 4$, and denote by $\sigma_4$ the projection on the first $5$ coordinates. We use that the first marginal of $\sigma$ is the Dirac mass of the predetermined point $a\in X$ in \cref{dynamical reform A} and the following:
\begin{enumerate}
    \item \label{intro point 1} An \textit{initialization property}; for an open set $E\subset X$, if $\mu(E)>0$, then  $$\sigma_4(X\times T^{-t}E \times T^{-t}E \times T^{-t}E  \times T^{-t}E)>0 ,$$ for some $t\in \N$ (see \cref{Initialization-step}).
    \item \label{intro point 2} \textit{Left-progressiveness} (see \cref{definition left-right progressive measure}); if $U_1,U_2,U_3\subset X$ are open sets and
    $\sigma_4(X\times U_1\times U_2\times U_3 \times X)>0 ,$ then $$\sigma_4(T^{-n}U_1\times (U_1 \cap T^{-n}U_2)\times (U_2\cap T^{-n}U_3)\times U_3 \times X)>0, $$
    for infinitely many $n\in \N$.
    \item \label{intro point 3} \textit{Right-progressiveness} (see \cref{definition left-right progressive measure}); if $U_1,U_2,U_3\subset X$ are open sets and
    $\sigma_4(X\times U_1\times U_2\times U_3 \times X)>0 ,$ then $$\sigma_4(X\times U_1 \times (U_2\cap T^{-n}U_1)\times (U_3\cap T^{-n}U_2) \cap T^{-n}U_3 )>0, $$
    for infinitely many $n\in \N$.
\end{enumerate}
The first two properties were already exploited in \cite{kmrr25}, where
the second was called \emph{progressiveness}. This notion is not sufficient 
for our purposes, and we refer to it as left-progressiveness, because we 
need to exploit its symmetrical version described in \textbf{\ref{intro point 3}}. 
It is this last property that enables us to enlarge the distance to the first 
coordinate, which ultimately allows us to increase the number of summands in the 
sumsets of \cref{dynamical reform A}.

To put things into perspective, for an open set $E\subset X$ with $\sigma_4(X\times X\times E\times X \times X) >0$,
a first application of left-progressiveness allows us to find $b_1\in \N$ such that 
\begin{equation} \label{eq ex intro 2}
    \sigma_4(X\times T^{-b_1} E \times E \times X \times X)>0,
\end{equation}
and with a second application of the same property, we find $b_2>b_1$ such that $$\sigma_4(T^{-(b_1+b_2)} E\times \left(T^{-b_1} E\cap T^{-b_2} E \right) \times E \times X \times X)>0.$$ 
Since the first marginal of $\sigma$ is $\delta_a$, we get that $T^{b_1+b_2}a\in  E$. Continuing this process, we get a sumset $\{ b_i+b_j : b_i,b_j\in B\text{ distinct} \}$ inside the set of return times $\{n\in \N: T^na\in E\}$ of $a$ to $E$. However, if we apply right-progressiveness beforehand in \eqref{eq ex intro 2}, we find $t\in \N$ such that
$$\sigma_4(X\times T^{-b_1}E\times (E \cap T^{-t - b_1}E)\times T^{-t}E \times X)>0. $$
In particular, $\sigma_4(X\times X\times T^{-b_1} \tilde E \times \tilde E \times X)>0$ where $\tilde{E}=T^{-t}E$ and $b_1 \in \N$ is the same as in \eqref{eq ex intro 2}. Therefore applying left-progressiveness multiple times as before, we find that not only $\{ b_i+b_j : b_i,b_j\in B\text{ distinct} \}$ is in the set $\{n\in \N: T^na\in E\}$, but also  $t+\{ b_i+b_j+b_m : b_i,b_j,b_m\in B\text{ distinct} \}$. 

The previous discussion gives a sense of how to obtain \cref{dynamical reform A} (and therefore \cref{theo-A}) from the initialization property, left-progressiveness, and right-progressiveness. Nevertheless, when it comes to \cref{theo-B}, this method is insufficient. The dynamical reformulation of this theorem is the following. 

\begin{theo}\label{dynamical reform B}
Let $(X,\mu,T)$ be an ergodic measure preserving system with a point $a\in X$ such that $\mu\left(\overline{\{T^na: n\in \N\}}\right)=1$. Let also $E\subset X$ be an open set with $\mu(E)>0$. Then, there exist an infinite set $B \subset \N$ and sequences $ (s_k)_{k\in \N}, (t_k)_{k\in \N}\subset \N$ such that 
\begin{equation} \label{eq rec for double shift equation}
    T^{\left(t_k+i s_k + b_{1}+\cdots + b_{i} \right)}a \in E,
\end{equation}
for all $k\in \N$, all $1\leq i \leq k$, and $b_1, b_2, \ldots, b_i $ distinct elements in $B$.
\end{theo}

In order to prove \cref{dynamical reform B} we need to use an additional 
property of our measure $\sigma \in \cM(X^{\N_0})$ 
which we term \textit{multiple 
right-progressiveness}. This notion is derived from appropriated double ergodic averages. To prove it, we first show an analog of Furstenberg's multiple recurrence theorem and then that for these related ergodic averages, the infinite-step pronilfactor is characteristic; the details can be found in \cref{sec mult recurrence}.

We refer to \cref{def multiple right progressive} for more details, but when, for instance, $k=4$, multiple right-progressiveness implies that for open sets $U_1,U_2\subset X$, if
$$\sigma_4(X\times U_1\times U_2 \times X \times X)>0 ,$$
then there are $t,s\in \N$ such that
$$\sigma_4(X\times U_1\times (U_2\cap T^{-(t+s)}U_1)\times ( T^{-(t+s)}U_2\cap  T^{-(t+2s)}U_1)\times (  T^{-(t+2s)}U_2 \cap  T^{-(t+3s)}U_1) )>0.$$

This property recovers both right-progressiveness and the initialization property (see \cref{new proof of initialization-step}). Multiple right-progressiveness not only serves to translate sets to the right, but it also allows us to consider an increasing number of patterns with increasing length, by paying the cost of adding an extra shift $s\in \N$, which corresponds to the second shift appearing in \cref{theo-B}. 

\subsection{Notation}\label{notational conventions}
By $\Z$, $\N_0$, $\N$ and $\R$ we denote, respectively, the integers,
nonnegative integers, positive integers (natural numbers), and the reals. 

Given a measure space $(X, \mathcal{B},\mu)$, a measurable space 
$(Y,\mathcal{C})$ and a measurable function $F: X \to Y$, we denote by 
$F\mu$ the push-forward of $\mu$ under $F$. That is, $F\mu(C)=\mu(F^{-1}C)$, 
for any $C\in \mathcal{C}$. Moreover, given a measurable transformation $T:X \to X$ and some integer $n\in \N$, we let $T^n$ denote the composition $T \circ T \circ \cdots \circ T$ taken $n$-times and we sometimes write $T^nf$ instead of $f\circ T^n$, where $f: X\to \C$ is a measurable map. For $k\in \N$ we write $\tilde{T}$ for the map $T\times T^2 \times \cdots \times T^k$, and $T_{\Delta}$ for the map $T^{\otimes k}=T\times T \times \cdots \times T$, where the product is taken $k$-times, omitting the dependence on $k$ which is clear from the context.  

If $S=\emptyset$ we let $T^{-S}E=X$. Also, if $S\subset \N$ is nonempty and $E\subset X$, we denote 
$$T^{-S}E=\bigcap_{n\in S} T^{-n}E. $$
In addition, for a finite set 
$F\subset \N$ we define $F^{\oplus 0}=\{0\}$, and for $i\in \N$ 
$$F^{\oplus i}=\begin{cases}
    \{f_1+\cdots +f_i :f_1,\ldots,f_i\in F \text{ distinct}\} & \text{ if } i\leq |F|\\
    \emptyset & \text{ otherwise}
\end{cases}.$$ 
Notice that, under this convention, $\emptyset^{\oplus 0}=\{0\}$. The motivation for introducing such notation is that if $A \subset \N$, $a \in X$ and $E \subset X$ are as in \cref{FCP} below, then for any nonempty set $S \subset \N$, $a \in T^{-S} E$ if and only if $S \subset A$. 
\subsection*{Acknowledgments}
We are grateful to Bryna Kra, Joel Moreira, Andreas Mountakis and Florian Richter for various helpful comments 
regarding a previous draft of this article. We also thank the organizers of 
Perspectives on Ergodic Theory and its Interactions at IMPAN, Warsaw and of the 
Summer School on Additive Combinatorics, Number Theory and Ergodic theory at 
EPFL, where part of this work was completed. The third author was partially supported by the National Science Foundation grant DMS-2348315.

\section{Negative results and optimality of  the main theorems}\label{counterexamples section}

In \cite{Kra_Moreira_Richter_Robertson_problems}, several possible results 
serving the purpose of a potential analog to Hindman's finite sums theorem 
were conjectured or asked. We begin with a weaker version of \cref{question kmrr intro} stated in the introduction.

\begin{question}[{\cite[Question 2.12]{Kra_Moreira_Richter_Robertson_problems}}] \label{KMRR Question 2.12}  
Let $A\subset \N$ with $\diff^*(A)>0$. Does there exist a nested sequence of infinite sets $B_1 \supset B_2 \supset B_3 \supset \dots$, where $B_{k+1}$ is cofinite in $B_k$ for each $k\in \N$, and a sequence of shifts $(t_k)_{k\in \N}$ such that, for any $k\in \N$ it holds that 
$$\bigg\{\sum_{n\in F}n: F\subset B_k\ \text{with}\ 1\leq |F|\leq k\bigg\}\subset A-t_k\ ?$$
\end{question}

Note that a positive answer to \cref{question kmrr intro} would clearly 
imply a positive answer to \cref{KMRR Question 2.12}. But one could -- ostensibly -- weaken these statements even further and in the setting of \cref{KMRR Question 2.12} simply ask for a nested sequence of infinite sets $B_1 \supset B_2 \supset B_3 \supset \dots$, without demanding they be cofinite in $B_1$. We remark that in fact this last statement is equivalent to that of \cref{KMRR Question 2.12} and a proof of this equivalence is included in \cref{cofinite and nested is the same} below. The next result, which follows by a modification of Straus's example mentioned in the introduction, disproves this statement. 
\begin{prop}\label{1.2 is optimal (ii)}
Given any $\epsilon>0$ there exists a set $A\subset \N$ with $\diff^*(A)>1-\epsilon$ such that for any sequence $(k_i)_{i \in \N}$, for any nested sequence $B_1\supset B_2 \supset \ldots$, of infinite sets, and any sequence $(t_i)_{i\in \N}$, it does not hold that
$$\bigg\{\sum_{n\in F}n: F\subset B_i\ \text{with}\ k_i\leq |F|\leq k_i+i \bigg\}\subset A-t_i,\ \text{for every}\ i\in \N.$$
\end{prop}

Before proving \cref{1.2 is optimal (ii)}, we list some 
consequences. 
\begin{remark}\label{counterexample 1}
    The set $A\subset \N$ provided by \cref{1.2 is optimal (ii)} is such that:
    \begin{enumerate}
        \item Taking $k_i=1$ for each $i\in \N$, then for any nested sequence $B_1 \supset B_2 \supset B_3 \supset \dots$, of infinite sets and any sequence of shifts $(t_k)_{k\in \N}$, it does not hold that 
$$\bigg\{\sum_{n\in F}n: F\subset B_k\ \text{with}\ 1\leq |F|\leq k \bigg\}
 \subset A-t_k,\ \text{for every}\ k\in \N.$$
In view of the discussion prior to the statement of \cref{1.2 is optimal (ii)}, this answers Questions \ref{question kmrr intro} and \ref{KMRR Question 2.12}. Note that it also shows the necessity of the auxiliary shift-sequence 
$(s_k)_{k\in \N}$ in the statement of \cref{theo-B}.
\item If we let $k_i=i$ for each $i\in \N$, then for any infinite set $B\subset \N$ and any sequence $(t_{k,\ell})_{k,\ell \in \N}$ it does not hold that
$$ \bigg\{\sum_{n\in F}n: F\subset B\ \text{with}\ k\leq |F|\leq k+\ell \bigg\} \subset A-t_{k,\ell},$$
for each $k,\ell \in \N$. This shows that \cref{theo-A} cannot be 
improved in the sense of allowing the infinite set $B$ to be independent of $\ell$. 
    \end{enumerate}
\end{remark}

\begin{proof}[Proof of \cref{1.2 is optimal (ii)}]
Let $\epsilon>0$ be fixed and consider an increasing sequence of prime numbers $(p_n)_{n\in \N}$ such that $\sum_{n=1}^{\infty} n/p_n <\epsilon.$ Then, we define the set 
\begin{equation} \label{eq set counterex}
    A=\N \setminus \bigcup_{n=1}^{\infty} \left( \bigcup_{j=0}^{n-1}p_n\N + j \right).
\end{equation}
By the choice of $(p_n)$ it readily follows that $\diff^*(A)>1-\epsilon$. Now, assume there exists an infinite sequence of sets $B_1 \supset B_2 \supset B_3 \supset \dots$ and a sequence of shifts $(t_i)_{i\in \N}$ such that for each $i\in \N$ it holds that
$$ \bigg\{\sum_{n\in F}n: F\subset B_i\ \text{with}\ k_i\leq |F|\leq k_i+i \bigg\} \subset A-t_i.$$
Notice that, if $r>i$, the fact that $B_{r} \subset B_i$ implies
$$\bigg\{\sum_{n\in F}n: F\subset B_r\ \text{with}\ k_i\leq |F|\leq k_i+i \bigg\} \subset A-t_i.$$
Fix $i,m,n,r\in \N$ with $t_{i}<n<p_n<r<m$ and $i<r<m$ so that $B_m \subset B_{r} \subset B_{i}$.

Assume first that there is an infinite set $\tilde{B}_m\subset B_m \cap  p_n\N$. In particular, $$\bigg\{\sum_{n\in F}n: F\subset \tilde{B}_m\ \text{with}\  |F|= k_i \bigg\} \subset p_n \N.$$  Then, by assumption and the above 
discussion, it must be the case that 
$$\bigg\{\sum_{n\in F}n: F\subset \tilde{B}_m\ \text{with}\  |F|= k_i \bigg\}+t_{i} \subset A.$$
However, $t_{i}<n$ and thus 
$$\bigg\{\sum_{n\in F}n: F\subset \tilde{B}_m\ \text{with}\  |F|= k_i \bigg\}+t_{i}  \subset \bigcup_{j=0}^{n-1} p_n\N + j \subset \N\setminus A,$$
which is a contradiction. 

Therefore, by the pigeonhole principle, we may assume there exists $c_n\in \{1,\ldots,p_n-1\}$ and an infinite set $\tilde{B}_{m} \subset B_m \cap (p_n\N+c_n)$. As before, the preceding discussion implies that
$$\bigg\{\sum_{n\in F}n: F\subset \tilde{B}_m\ \text{with}\ k_r\leq |F|\leq k_r+r \bigg\} \subset A-t_{r}.$$
Now, $p_n$ is a prime and $c_n \not\equiv 0 \pmod {p_n}$, hence $c_n$ is a generator in $\Z_{p_n}$. As $r>p_n$ this means that there exists some $d\in \{k_{r},\ldots,k_{r}+r\}$ such that $dc_n+t_{r} \equiv 0 \pmod{p_n}.$ 
We have thus established that 
$$\bigg\{\sum_{n\in F}n: F\subset \tilde{B}_m\ \text{with}\ |F|=d \bigg\}+t_r \subset A \cap \left(p_n\N + d c_n+t_{r}\right) \subset A \cap p_n\N = \emptyset,$$
which is again a contradiction. 
\end{proof}

\begin{prop}
There exists a set $A\subset\N$ with $\diff^*(A)>0$ such that the statement of \cref{theo-B} fails if either of the sequences $(t_k)_{k\in \N}$ or $(s_k)_{k\in \N}$ is bounded.     
\end{prop}
\begin{remark*}
The analogous is true for the statement of \cref{theo-A}.  
\end{remark*}

\begin{proof}
Take $A\subset \N$ as in \eqref{eq set counterex}.
Suppose \cref{theo-B} holds with a bounded 
sequence of shifts $(t_k)_{k\in \N}$ and some sequence $(s_k)_{k\in \N}$. Let $n>q=\max\{t_k:k\in \N\}$. Assume that for some $k>p_n$ there exists an infinite set $\tilde{B}+s_k\subset (B+s_k)\cap (p_n\N+d_n)$ for some $d_n\in \{1,\ldots,p_n-1\}$. Then, as $d_n$ is a generator in $Z_{p_n}$ and $k>p_n$ we can find $d\in \{1,\ldots,k\}$ so that $d\cdot d_n +t_k\equiv 0 \pmod{p_n}$. We thus reach a contradiction because
$$ \bigg\{\sum_{n\in F} n : F\subset \tilde{B}+s_k\ \text{with}\ |F|=d\bigg\}+t_k \subset p_n\N+d\cdot d_n+t_k \subset p_n\N \subset \N\setminus A. $$

Hence, for every $k>p_n$, there is an infinite set $\tilde{B}+s_k\subset (B+s_k)\cap p_n\N$. This leads to a contradiction once again, since $\{\sum_{n\in F} n : F\subset \tilde{B}+s_k\ \text{with}\ |F|=k\}+t_k \subset A$ and
$$\bigg\{\sum_{n\in F} n : F\subset \tilde{B}+s_k\ \text{with}\ |F|=k\bigg\}+t_k \subset p_n\N+t_k \subset \bigcup_{j=0}^q p_n\N+j \subset \bigcup_{j=0}^n p_n\N+j \subset \N\setminus A.$$

To see that 
\cref{theo-B} does not in general hold with a finite sequence 
$(s_k)_{k\in \N}$, note that, if it did we would find by the pigeonhole principle an infinite set $B$, an integer $s\in \N$ and sequences $(n_k)_{k\in \N},(t_k)_{k\in \N}$ so that 
$$\bigg\{\sum_{n\in F} n: F\subset B+s\ \text{with}\ 1\leq |F| \leq n_k\bigg\} \subset A-t_k,\ \text{for every}\ k\in \N.$$
This would contradict our negative answer to \cref{question kmrr intro}.    
\end{proof}

We next show that imposing a cofiniteness assumption for nested versions of 
infinite sumsets results, as, for instance, in \cref{KMRR Question 2.12}, is 
superfluous in the following sense. 

\begin{prop}\label{cofinite and nested is the same}
If $A \subset \N$, then the following are equivalent:
\begin{enumerate}[(i)]
    \item \label{nested 1}There exists a nested sequence of infinite sets $B_1 \supset B_2 \supset B_3 \supset \dots$, where $B_{k+1}$ is cofinite in $B_k$, for each $k\in \N$, and a sequence of shifts $(t_k)_{k\in \N}$ such that, for any $k\in \N$ 
\begin{equation} \label{eq nested things}
    \bigg\{\sum_{n\in F}n: F\subset B_k\ \text{with}\ 1\leq |F|\leq k\bigg\}\subset A-t_k\
\end{equation}
\item \label{nested 2}There exists a nested sequence of infinite sets $B_1 \supset B_2 \supset B_3 \supset \dots$, and a sequence of shifts $(t_k)_{k\in \N}$ such that, \eqref{eq nested things} holds for any $k\in \N$. 
\end{enumerate}
\end{prop}

\begin{proof}
    It is clear that \eqref{nested 1} implies \eqref{nested 2}. For the converse, suppose $(B_i)_{i\in \N}$ satisfies \eqref{nested 2}. Let $b_1=\min\{b: b \in B_1\}$ and then, let $b_2=\min\{b: b\in B_2\setminus \{b_1\}\}$. Inductively, we can choose $b_{i+1}=\min\{b: b\in B_{i+1}\setminus \{b_1,\ldots,b_i\}\}$. With that we have built a sequence $(b_i)_{i \in \N}$ which is infinite because every set $B_i$ is infinite. For each $i\in \N$, define $C_i = \{b_j \colon j \geq i\}$ and observe that, since $C_i \subset B_i $, the sequence $C_1\supset C_2 \supset C_3 \supset \ldots$ fulfills \eqref{eq nested things}. Finally, by construction, $C_i \backslash C_{i+1} = \{ b_{i+1}\}$, hence $C_{i+1}$ is cofinite in $C_{i}$, concluding the proof.
\end{proof}

\begin{remark}
    The same proof can be performed for other types of sumsets results. In the present work, our ``positive'' results are proven in non-nested versions, that is, the strongest possible. We do not know, however, if  \cref{theo-A} and \cref{theo-B} are equivalent to their respective nested versions. It could be the case, for example, that every element $b \in B_1$ appears only in finitely many $B_i$ which a priori makes the construction of a single infinite set $B$ impossible. 
\end{remark}

We finish this section by showing that infinite analogs of \Erdos\ 
progressions (see \cite[Definition 2.1]{kmrr25}) can be utilized 
to obtain weaker versions of our main results, \cref{theo-A} and 
\cref{theo-B}. The 
sumsets obtained via this method are generated by a nested sequence of infinite sets.

\begin{defn}\label{Infinite EP}
Let $X$ be a compact metric space and $T \colon X \to X$ a homeomorphism. A sequence $(x_k)_{k\in \N_0}$ in $X$ is 
called an infinite \Erdos\ progression if there exists an increasing sequence 
$(c_n)_{n\in \N}$ of natural numbers such that 
$$\lim_{n\to \infty}T^{c_n}x_{k}=x_{k+1},\ \text{for every}\ k\in \N_0.$$
\end{defn}

\begin{prop}\label{example appendix Erdos progressions}
Let $X$ be a compact metric space, $T \colon X \to X$ a homeomorphism and $(U_k)_{n\in \N}$ a sequence of open 
subsets of $X$. Assume there exists an infinite \Erdos\ progression 
$(x_k)_{k\in \N_0}$ such that $x_k\in U_k$ for each $k\in \N$. Then, there    exists a sequence of infinite sets $B_1,B_2,\ldots$ such that 
$B_{k+1} \subset B_k$ and
$$\Big\{ \sum_{n\in F}n \colon F\subset B_k,\ |F|=k\Big\} \subset \Big\{ m \in \N \colon T^mx_0\in U_k\Big\},$$
for every $k\in \N$. 
\end{prop}

\begin{proof}
Let $(c_n)_{n\in \N}$ be a sequence such that the convergence in \cref{Infinite EP} is satisfied. Fixing $k\in \N$, the proof of \cite[Lemma 2.2]{kmrr25} gives an infinite set $B_k \subset \{c_n: n\in \N\}$ such that 
$$\bigg\{ \sum_{n\in F}n : F\subset B_k,\ |F|=k\bigg\} \subset \bigg\{m \in \N \colon T^mx_1\in U_k\bigg\}.$$
Then, writing $B_k=\{b_{k,1}<b_{k,2}<\cdots\}$, since $(b_{k,n})_{n\in \N}$ is 
a subsequence of $(c_n)_{n\in \N}$ we have 
$$\lim_{n\to \infty}T^{b_{k,n}}x_{m}=x_{m+1},\ \text{for every}\ m\in \N.$$
Another application of \cite[Lemma 2.5]{kmrr25} then gives an infinite set $B_{k+1} \subset \{b_{k,n}:n\in \N\}$ so that 
$$\bigg\{ \sum_{n\in F}n : F\subset B_{k+1},\ |F|=k+1\bigg\} \subset \bigg\{m \in \N \colon T^mx_1\in U_{k+1}\bigg\}.$$
Continuing inductively we obtain the result. In fact, the proof 
of \cref{cofinite and nested is the same} shows that we can actually 
choose the sets $B_1,B_2,\ldots,$ so that $B_k\setminus B_{k+1}$ is a finite set 
for each $k\in \N$. 
\end{proof}

\section{Dynamical interpretation} \label{reduction to dynamics and outline of proofs}

\subsection{Classical terminology}\label{Terminology}

If $X$ is a compact metric space and $T:X \to X$ is a homeomorphism, we call
$(X,T)$ a \emph{topological system}. In this context, if $\mu$ is a Borel 
probability measure that is invariant under $T$, i.e., $T\mu=\mu$, 
we call $(X,\mu,T)$ a \emph{measure preserving system} or a 
\emph{dynamical system} or simply a \emph{system}. A system is 
\emph{ergodic} if $T^{-1}A=A$ implies that  $\mu(A)\in \{0,1\},$ for any Borel set $A\subset X$.

A \Folner\ sequence $\Phi=(\Phi_N)_{N\in \N}$ in $\N$ is a sequence of nonempty finite subsets of $\N$ which are asymptotically invariant under any shift, that is, 
$$\lim_{N\to \infty} \frac{|\Phi_N \cap (\Phi_N+t)|}{|\Phi_N|}=1,$$
for every $t\in \N$. Observe that a set $A\subset \N$ has positive upper Banach density if and only if there is a \Folner\ sequence $\Phi=(\Phi_N)_{N\in \N}$ such that 
$$ \lim_{N\to \infty} \frac{|A\cap \Phi_N|}{|\Phi_N|}>0.$$
Given a measure preserving system $(X,\mu,T)$, and a \Folner\ sequence 
$\Phi=(\Phi_N)_{N\in \N}$ in $\N$, we say that a point $a\in X$ is generic 
for $\mu$ along the \Folner\ sequence $\Phi$, and we write $a\in \gen(\mu,\Phi)$, if 
$$\mu=\lim_{N\to \infty} \frac{1}{|\Phi_N|} \sum_{n\in \Phi_N} \delta_{T^na}, $$
where the limit is in the weak* topology and by $\delta_x$ we denote the Dirac mass at $x\in X$. 

Let $(X, \mu,T)$ and $(Y,\nu,S)$ be measure preserving systems. We say that $(Y, \nu,S)$ (or simply $Y$) is a \emph{factor} of $(X, \mu,T)$ (or simply $X$) if there exists a measurable map $\pi \colon X \to Y$, called a \emph{factor map}, such that $\pi\mu = \nu$ and $\pi \circ T = S \circ \pi$, up to null-sets. In this case $X$ is called an \emph{extension} of $Y$. 
If, in addition, the function $\pi \colon X \to Y$ is continuous and 
surjective, then it is called a \emph{continuous factor map}. In an abuse of notation we often use the same letter for the transformation of a factor and its extension, that is, we write $(Y,\nu,T)$. Given a factor map $\pi: X\to Y$ and a function $f\in L^2(\mu)$, 
we let $\mathbb{E}(f | Y)$ denote the conditional expectation 
$\mathbb{E}(f | \pi^{-1}\mathcal{B}(Y))$, where $\mathcal{B}(Y)$ is the Borel 
$\sigma$-algebra on $Y$.

Let $(X_j,\mu_j,T_j)_{j\in \N}$ be a sequence of measure preserving systems endowed with factor maps $\pi_{j,j+1}: X_{j+1}\to X_j$ for each $j\in \N$. An \emph{inverse limit} of this sequence of systems is the unique system $(X,\mu,T)$ (up to isomorphism) endowed with the factor maps $\pi_j:X\to X_j$, $j\in \N$ such that
\begin{enumerate}[(i)]
    \item \label{point inverse 1}  $\pi_{j+1}=\pi_{j,j+1} \circ \pi_j$, for every $j\in \N$ and 
    \item For $1\leq p <\infty$, $\bigcup_{i\in \N} \{f \circ \pi_i : f\in L^p(\mu_i)\}$ is dense in $L^p(\mu)$.
\end{enumerate}
A similar definition applies for topological dynamical systems (replacing factor maps by continuous factor maps and integrable functions by continuous ones). For existence and alternative definitions of inverse limits we refer the reader to \cite[Chapter 6.2]{Host_Kra_nilpotent_structures_ergodic_theory:2018}.

\subsection{Translation to dynamics}\label{translation to dynamics and some reductions}

The proofs of our main combinatorial 
results 
begin by translating them to statements in ergodic theory. A fundamental step in order to facilitate 
this reformulation is a version of  Furstenberg's correspondence 
principle introduced in \cite{Furstenberg77}. 
For a proof of this version, we refer the reader to 
\cite{Kra_Moreira_Richter_Robertson:2022}.

\begin{prop}\label{FCP}
Let $A\subset \N$ be a set of with $\diff^*(A)>0$. Then, there 
exists an ergodic system $(X,\mu,T)$, a point $a\in X$ with 
$\mu\left(\overline{\{T^na: n\in \N\}}\right)=1$ and a clopen set 
$E\subset X$ such that $A=\{n\in \N: T^na\in E\}$ and $\mu(E)>0$.
\end{prop}

We show the deduction of our main combinatorial results, Theorems 
\ref{theo-A} and \ref{theo-B}, from the dynamical theorems mentioned in the introduction.

\begin{proof}[Proof that Theorem \ref{dynamical reform B} implies Theorem \ref{theo-B}]
Let $A\subset\N$ be a set with $\diff^*(A)>0$. By the correspondence principle (\cref{FCP}) we find an ergodic system $(X,\mu,T)$, a point $a\in X$ with 
$\mu\left(\overline{\{T^na: n\in \N\}}\right)=1$ and a clopen set 
$E\subset X$ such that $A=\{n\in \N: T^na\in E\}$ and $\mu(E)>0$. Theorem \ref{dynamical reform B} then guarantees the existence of an infinite set $B\subset \N$ and sequences $(s_k)_{k\in \N}, (t_k)_{k\in \N}\subset \N$ such that 
$$T^{\left(t_k+i s_k + b_1+\cdots + b_i \right)}a \in E,$$
for every $k\in \N$ and for every $1\leq i \leq k$, and $b_1,b_2,\ldots,b_i$ distinct elements of $B$. Since $A=\{n\in \N: T^na\in E\}$ we have that 
$$b_1+\cdots + b_i+is_k \in A-t_k,$$
for every $k\in \N$ and for every $1\leq i \leq k$, and $b_1,b_2,\ldots,b_i$ distinct elements of $B$. This translates to 
 $$ \bigg\{ \sum_{n\in F} n \colon \ F\subset B+s_k \text{ with } 1\leq |F|\leq k \bigg\} \subset A-t_k ,\ \text{for every}\ k\in \N,$$
which is precisely the statement of \cref{theo-B}.
\end{proof}

The proof that Theorem \ref{dynamical reform A} implies Theorem 
\ref{theo-A} is analogous to the previous and as such we omit it. 
Let us also show that \cref{theo-A}, for example, implies 
\cref{coro-C}. 

\begin{proof}[Proof that \cref{theo-A} implies \cref{coro-C}]
    Let $B \subset \N$ and $(t_k)_{k\in \N}\subset \N$ be the infinite set and the sequence obtained by \cref{theo-A} for $\ell=0$, so that 
    \begin{equation} \label{eq theo with l=0}
        \left \{\sum_{b \in F} b : F \subset B, |F| =k \right\}\subset A-t_k,\ \text{for every}\ k\in \N.
    \end{equation}
    Let $B=\{b_k:k\in \N\}$ be an enumeration of the set $B$, let $\{p_k: k\in \N\}$ be an enumeration of the primes and for each $k\in \N$ define $B_k= \{b_{p_k^n} + t_{k}-t_{k-1} \colon k \in \N\}$, where $t_0=0$. By \eqref{eq theo with l=0} we have that
    \begin{equation*}
        \sum_{i=1}^k (b_{p_i^{n_i}} + t_{i}-t_{i-1}) = \Big( \sum_{i=1}^k b_{p_i^{n_i}}  \Big) + t_k \subset A,
    \end{equation*}
    for every $k\in \N$ and any $n_1,\ldots,n_k\in \N$. By definition, this implies that $B_1 + B_2 + \cdots + B_k\subset A,$ for every $k\in \N$, concluding the proof.
\end{proof}

We saw that the dynamical results, Theorems \ref{dynamical reform A} and \ref{dynamical reform B}, imply respectively Theorems \ref{theo-A} and \ref{theo-B} -- and hence, also, \cref{coro-C}. Their proofs are then the main focus of the rest of this paper. 

\section{Progressive measures and infinite-step nilsystems} \label{section measures and nilsystems}

\subsection{Pronilfactors}

As in \cite{Kra_Moreira_Richter_Robertson:2022} and \cite{kmrr25}, 
(also \cite{host2019short} and 
\cite{Kra_Moreira_Richter_Robertson:2023} for specific use of the Kronecker factor)
we use measures on pronilfactors to generate the infinite sumsets. 
A $k$-step nilsystem $(X,\mu,T) = (G/\Gamma,\mu,T)$ is a dynamical system given by a $k$-step nilpotent Lie group $G$, with $\Gamma$ being a co-compact discrete subgroup, $T \colon X \to X$ denoting the left translation by a fixed element of $G$ and $\mu$ being the Haar measure of $X$. For the $k$-step nilsystem $(X,\mu,T)$ the properties of ergodicity, minimality, transitivity and unique ergodicity are equivalent. Moreover, if $x \in X$ is any point, the orbit closure $\overline {\{ T^n x \colon n \in \Z \}} $ is a minimal and uniquely ergodic $k$-step subnilsystem. An inverse limit (in the measure theoretic sense) of $k$-step nilsystems is called a $k$-step pronilsystem. 
All the above properties are preserved under inverse limits of $k$-step nilsystems, so they also hold for $k$-step pronilsystems. For these classical facts and other details we refer the reader to \cite[Chapter 11 and Chapter 13]{Host_Kra_nilpotent_structures_ergodic_theory:2018}.  

\begin{defn}
A system $(Z,m,T)$ is called an \emph{infinite-step pronilsystem} if it is 
an inverse limit of a sequence $\left((Z_k, m_k,T)\right)_{k\in \N}$, where each $(Z_k, m_k, T)$ is a $k$-step pronilsystem. 
\end{defn}
Infinite-step pronilsystems share many properties with finite-step pronilsystems; some of those, which are exploited in this paper, are listed in the end of this subsection. 

Let $(X, \mu,T)$ be an ergodic measure preserving system. Then, for every $k \in \N$, $(X, \mu,T)$ has a maximal factor isomorphic to a $k$-step pronilsystem, called the $k$\emph{-step pronilfactor} (see \cite[Chapter 16]{Host_Kra_nilpotent_structures_ergodic_theory:2018}). We shall denote this factor by $(Z_k,m_k,T)$. The inverse limit of $(Z_k,m_k,T)$ is the \emph{infinite-step pronilfactor of} $X$ and we denote it by $(Z_{\infty}, m,T)$. The infinite-step pronilfactor of $X$ coincides with the maximal factor of $X$ isomorphic to an infinite-step pronilsystem. In their seminal paper \cite{Host_Kra_nonconventional_averages_nilmanifolds:2005}, Host and Kra  showed that it is possible to describe these factors by studying uniformity seminorms, classically referred to as Gowers-Host-Kra seminorms.  

\begin{defn}
    Let $(X, \mu,T)$ be an ergodic system and $f \in L^{\infty}(\mu)$. The $k$-uniformity seminorms of $f$, $\norm{f}_{U^{k}(X,\mu,T)}$, are defined inductively as follows: 
    \begin{equation*}
        \norm{f}_{U^0(X, \mu,T)} =  \int f d\mu 
    \end{equation*}
    and 
    \begin{equation} \label{eq def seminorm}
        \norm{ f}_{U^{k+1}(X, \mu,T)}^{2^{k+1}}= \lim_{H \to \infty }\frac{1}{H} \sum_{h \leq H} \norm{f \cdot \overline{T^hf}}_{U^k(X, \mu,T)}^{2^k}.
    \end{equation}
\end{defn}

\begin{remark*}
    The function $\norm{ \cdot }_{U^{k}(X, \mu,T)}$ defines a seminorm when $k \geq 1$. The fact that the limit in \eqref{eq def seminorm} exists is proved in \cite{Host_Kra_nonconventional_averages_nilmanifolds:2005} by iterative applications of the mean ergodic theorem. In the same paper, it is also shown that $\norm{ \cdot }_{U^{k+1}(X, \mu,T)}$ is actually a norm for $k$-step pronilsystems. 
\end{remark*}

The main theorem in \cite{Host_Kra_nonconventional_averages_nilmanifolds:2005} shows that the $k$-step pronilfactor, $Z_{k}$, of an ergodic system $(X, \mu,T)$ is characterized by the property 
    \begin{equation*} 
        \norm{f}_{U^{k+1}(X, \mu,T)} = 0 \iff \E(f \mid Z_{k}) =0,
    \end{equation*}
    for any $f\in L^{\infty}(\mu)$. Similarly, for $Z_{\infty}$ we have the following characterization.

    \begin{prop} \label{infinite step nilfactor and HK seminorms}
    Let $(X,\mu,T)$ be an ergodic system and $(Z_{\infty},m,T)$ denote its infinite-step pronilfactor. For any $f \in L^{\infty}(\mu)$ we have that 
        \begin{equation*}
            \lVert f - \E(f \mid Z_{\infty}) \rVert_{U^k(X,\mu,T)} = 0 \quad \text{ for all } k \in \N.
        \end{equation*}
    Conversely, if $h\in L^{\infty}(\mu)$ is such that $\norm{h}_{U^k(X,\mu,T)}=0$ for all $k\in \N$, then $\mathbb{E}(h | Z_{\infty})=0$. 
    \end{prop}

    \begin{proof} By the definition of $Z_{\infty}$, we have $\displaystyle \E(f \mid Z_{\infty}) = \E(f \mid Z_s)$ in $L^2(\mu)$. For  $k \in \N$ fixed,
        \begin{equation*}
            \lVert f - \E(f \mid Z_s) \rVert_{U^k(X,\mu,T)} \leq \lVert f - \E(f \mid Z_s) \rVert_{U^{s+1}(X,\mu,T)} = 0 \quad \text{ for all } s \geq k.
        \end{equation*}
        Thus $\displaystyle \lVert f - \E(f \mid Z_{\infty}) \rVert_{U^k(X,\mu,T)} = \lim_{s \to \infty} \lVert f - \E(f \mid Z_s) \rVert_{U^k(X,\mu,T)} =0$.

        For the converse, assume $\mathbb{E}(h | Z_{\infty}) \neq 0$ in $L^2(\mu)$. Then, since $\mathbb{E}(h | Z_{\infty})= \lim_{s\to \infty} \mathbb{E}(h | Z_s)$ in $L^2(\mu)$, there must exist $k\in \N$ so that $\mathbb{E}(h | Z_k) \neq 0$ in $L^2(\mu)$. However this contradicts the assumption that $\norm{h}_{U^{k+1}(X,\mu,T)}=0$.
    \end{proof}
A property that we use throughout the article is that of unique ergodicity of 
orbit closures in infinite-step pronilsystems.    
\begin{prop} \label{lemma Zd action in infinite set pronilsystem} \label{appendix lemam to interchange limits}
    Let $(Z, m,T)$ be an infinite-step pronilsystem, $d \geq 1$ and $\Phi=(\Phi_N)_{N\in \N}$ a F\o lner sequence in $\Z^d$. Then, for every $k \geq 2$, any $x_1, \ldots, x_k \in Z$ and any non-trivial $S_i = T^{a_{i,1}} \times T^{a_i,2} \times \cdots \times T^{a_{i,k}}$ for $i =1, \ldots , d$, the measure 
    \begin{equation*}
        \rho =  \lim_{N \to \infty} \frac{1}{|\Phi_N|} \sum_{(n_1, \dots, n_d) \in \Phi_N} (S_1^{n_1} \cdots  S_d^{n_d}) \delta_{(x_1,\dots,x_k)},
    \end{equation*}
    where the limit is in the weak* topology, is well-defined. Moreover $\rho$ is independent from the F\o lner sequence $\Phi$. 
\end{prop}

\begin{proof}
    Let $Z_s$ be the $s$-step pronilfactor of $Z$ with factor map $\pi_s \colon Z \to Z_s$. It follows by \cite{Host_Kra_nonconventional_averages_nilmanifolds:2005} that $\pi_k \colon Z \to Z_s$ is continuous. Let $\Omega_s$ be the orbit closure of $\boldsymbol (\pi_s(x_1), \pi_s(x_2), \ldots, \pi_s(x_k)) \in Z_{s}^{k}$ under the $\Z^d$-action given by the maps $S_1, \ldots, S_d $. By properties of pronilsystems, $(\Omega_s, S_1, \ldots, S_d)$ is minimal and uniquely ergodic. If $\Omega_{\infty}$ is the orbit closure of  $\boldsymbol (x_1, x_2, \ldots, x_k) \in Z^{k}$ under the $\Z^d$-action given by the maps $S_1, \ldots, S_d $, then $\Omega_{\infty}$ is the inverse limit of $(\Omega_{s})_{s\in \N}$ and therefore is minimal and uniquely ergodic which concludes the proposition. 
    \end{proof}

\subsection{Progressive measures}

We now define the notion of progressiveness for measures on 
finite (and infinite) product spaces. This property can be 
viewed as a variation of recurrence. Later, we   introduce a more general property, which can analogously be viewed 
as variation of multiple recurrence. These properties are the key ingredients that allow us to prove our main dynamical theorems.

\begin{defn} \label{definition left-right progressive measure}
    Let $(X,T)$ be a topological system and let $k\in \N$. We say that a Borel probability measure $\tau \in \cM (X^{k+1})$ is \emph{left-progressive} if for all open sets $U_1, \ldots , U_k \subset X$ with 
    $$ \tau(X \times U_1 \times \cdots \times U_k) >0,$$
there exist infinitely many $n \in \N$ such that
\begin{equation*} 
    \tau((X \times U_1 \times \cdots \times U_k) \cap (T \times  \cdots \times T)^{-n} (U_1 \times \cdots \times U_k \times X)) >0.
\end{equation*}
We say that $\tau$ as above is \emph{right-progressive} if for all open sets $U_1, \ldots , U_{k-1} \subset X$ with 
$$\tau(X \times U_1 \times \cdots \times U_{k-1} \times X) >0,$$ 
there exist infinitely many $n \in \N$ such that
\begin{equation*} 
    \tau((X \times U_1 \times \cdots \times U_{k-1} \times X) \cap (T \times  \cdots \times T)^{-n} (X \times X \times U_1 \times \cdots \times U_{k-1})) >0.
\end{equation*}
Finally, a measure is \emph{progressive} if it is both left- and right-progressive. 
\end{defn}

For some measures on product spaces that we   consider, 
the first marginal has a distinguished role from the rest. 
As a useful convention we   index such measures starting 
from the $0$-th coordinate. 

\begin{remark*}
    In \cite[Definition 3.1]{kmrr25}, the authors introduced a notion of progressive measures that coincides with what we refer to as left-progressive.
\end{remark*}

We naturally extend this notion to measures $\tau \in \cM(X^{\N_0})$.

\begin{defn} \label{definition progressive measures in XN}
    Let $(X,T)$ be a topological system.  We say that a Borel probability measure $\tau \in \cM(X^{\N_0})$ is \emph{progressive} if $P_k \tau$ is progressive for all $k \geq 2$, where $P_k \colon X^{\N_0} \to X^{k+1}$ is the projection onto the first $k+1$ coordinates (including the $0$-th coordinate). 
\end{defn}

We now proceed to define a Borel measure $\sigma$ on the infinite product
space $X^{\N_0}$, with base space coming from a given measure preserving 
system $(X,\mu,T)$, which   turn out to be progressive. In order to 
achieve this we   first define a measure $\xi$ on $Z_{\infty}^{\N}$, where $(Z_{\infty},m,T)$ 
is the infinite-step pronilfactor of $(X,\mu,T)$ and then lift this measure 
on $X^{\N_0}$. We remark that the definition of $\sigma$ crucially depends 
on a single predetermined point $a\in X$. However, the factor map $\pi: X\to Z_{\infty}$, which is used to lift $\xi$ to $\sigma$, is a priori only measurable and thus defined almost everywhere (hence, potentially, not on $a$). Fortunately, we can assume that the factor map is actually continuous, a manoeuvre known among aficionados that was also exploited in \cite{kmrr25}. This reduction is formally proven at the end of this section.

\begin{defn} \label{def top infinite step pronil}
We say that $(X,\mu,T)$ has a \emph{topological infinite-step pronilfactor} if there is a continuous factor map $\pi_{\infty}: X \to Z_{\infty}$, where $(Z_{\infty},m,T)$ denotes the infinite-step pronilfactor of $(X,\mu,T)$.  
\end{defn}

Let $(X, \mu,T)$ be an ergodic system with topological infinite-step 
pronilfactor $(Z_{\infty},m,T)$. Let also $\Phi=(\Phi_N)_{N\in \N}$ be 
a F\o lner sequence in $\N$ and $a \in \gen(\mu,\Phi)$. Similarly to \cite{kmrr25}, our goal is to construct a measure $\xi \in \cM(Z_{\infty}^{\N})$ given by
\begin{equation} \label{eq def xi}
   \xi =  \lim_{N \to \infty} \frac{1}{|\Phi_N|} \sum_{n \in \Phi_N} \prod_{i=1}^{\infty}\delta_{T^{in}\pi_{\infty}(a)},
\end{equation}
where the limit (assuming it exists) is in the weak* topology, then lift it and build a progressive measure $\sigma \in \cM(X^{\N_0})$ (see \eqref{eq definition sigma}). For that we first need to ensure that the limit in \eqref{eq def xi} is well-defined.

\begin{lemma} \label{xi_k is well-defined}
    Let $(Z, m,T)$ be an infinite-step pronilsystem, $a \in Z$ and $\Phi=(\Phi_N)_{N\in \N}$ a F\o lner sequence in $\N$. Then for every $k \geq 2$, the measure 
    \begin{equation*}
        \xi_k =  \lim_{N \to \infty} \frac{1}{|\Phi_N|} \sum_{n \in \Phi_N} \delta_{T^n a} \times \delta_{T^{2n}a} \times \cdots \times  \delta_{T^{kn}a}
    \end{equation*}
    is well-defined. Moreover $\xi_k$ is independent from the F\o lner sequence $\Phi$. 
\end{lemma}

\begin{proof}
This is a corollary of \cref{appendix lemam to interchange limits} with $d=1$ and $S_1=T\times T^2 \times \cdots \times T^k$.
\end{proof}

By a classical approximation argument we deduce that the limit in \eqref{eq def xi} exists.

\begin{cor}
    The measure given by \eqref{eq def xi} is well-defined. 
\end{cor}
\begin{proof}
This follows at once from \cref{xi_k is well-defined} using the definition of weak* convergence, the observation that $P_k^*\xi_{n}=\xi_k$, for each $k,n\in \N$, with $k\leq n$, where $P_k^*$ denotes the projection onto the first $k$ coordinates, and the fact that $\{f_1 \otimes \dots \otimes f_k : k\in \N, f_1,\ldots,f_k \in C(Z_{\infty})\}$ is densely embedded in $C(Z_{\infty}^{\N})$. 
\end{proof}

With that we can define the Borel probability measure $\sigma \in \cM(X^{\N_0})$ through the use of Kolmogorov's extension theorem as follows. 
\begin{defn}\label{definition of sigma}
    For $k\in \N$, we define the measure $\sigma_k\in \cM(X^{k+1})$ via the formula 
\begin{equation} \label{eq definition sigma}
    \int_{X^{k+1}} f_0 \otimes  f_1 \otimes \cdots \otimes f_k\ \diff \sigma_k = f_0(a) \int_{Z^{k}_{\infty}}  \E( f_1 \mid Z_{\infty}) \otimes \cdots \otimes \E( f_k \mid Z_{\infty})\ \diff P_k^* \xi,
\end{equation}
for all $f_0,f_1, \ldots, f_k \in C(X) $, where $P_k^*$ denotes the projection onto the coordinates $\{1,\ldots,k\}$ (excluding the $0$-th coordinate). It follows by Stone-Weierstrass' theorem that each $\sigma_k$ is well-defined via \eqref{eq definition sigma}. The sequence of measures $(\sigma_k)_{k\in \N}$ satisfies the compatibility condition of Kolmogorov's extension theorem, so we denote by $\sigma \in \cM(X^{\N_0})$ the unique extension of this sequence. 
\end{defn}

\begin{remark*}
    By definition, for any $k\in \N$, $P_k\sigma=\sigma_k$, where $P_k$ denotes the projection onto the first $k+1$ coordinates, including the $0$-th coordinate. It follows immediately from the definition that $P_k\sigma$ is invariant under $Id\times T\times \cdots \times T^k$. Moreover, for each $k\in \N$ and for each $i=1,\ldots,k$ we have that the marginals $P_k\sigma_i$ of $P_k\sigma$ satisfy $P_k\sigma_i \leq i\mu$, i.e., $P_k\sigma_i (A)\leq i\mu(A)$ for any Borel set $A\subset X$. This follows using the unique ergodicity of orbits on $(Z_{\infty},T)$ and the $T^i$-invariance of the marginals $P_k^*\xi_i$ of $P_k^*\xi$ as in \cite[Lemma 4.6]{kmrr25}.
\end{remark*}
In order to prove Theorems \ref{dynamical reform A} and \ref{dynamical reform B} we first work under the assumption of 
topological infinite-step pronilfactors. This allows us to define the measure 
$\sigma$ as above, and using this we   prove the following theorems.

\begin{theo}\label{topological pronilfactors 2}
Let $(X,\mu,T)$ be an ergodic system with topological infinite-step pronilfactor. Let $\Phi=(\Phi_N)_{N\in \N}$ be a \Folner\ sequence, let 
$a\in \gen(\mu,\Phi)$ and let $\sigma$ be the measure defined in \cref{definition of sigma}. Then, for $\ell\in \N$ fixed and any
open set $E\subset X$ with $\mu(E)>0$ there exist sequences $(b_n)_{n\in \N}, (t_k)_{k\in \N}\subset \N$ such 
that
\begin{equation*}
       0<\sigma_{k+\ell+1}\left( \bigcap_{i=1}^{k+1}  \prod_{j=0}^{k+\ell+1}\bigcap_{h=\max(j,i)}^{i+\ell}  T^{-\left( t_i+ B_k^{\oplus h-j} \right)}  E   \right),  
   \end{equation*}
for every $k\in \N$, where $B_k=\{b_1,\ldots,b_{k}\}$.
\end{theo}

\begin{theo}\label{topological pronilfactors 1}
Let $(X,\mu,T)$ be an ergodic system with topological infinite-step 
pronilfactor. Let $\Phi=(\Phi_N)_{N\in \N}$ be a \Folner\ sequence, 
let 
$a\in \gen(\mu,\Phi)$ and let $\sigma$ be the measure defined in \cref{definition of sigma}. Then, 
for every open set $E\subset X$ with $\mu(E)>0$ there exist 
sequences $(b_n)_{n\in \N}, (s_k)_{k\in \N}, (t_k)_{k\in \N}\subset \N$ such 
that 
 $$
\sigma_{k+1}\left(\bigcap_{i=1}^k  \prod_{j=0}^{k+1}  \bigcap_{\ell=\max(j,1)}^{i+1}  T^{-\left( t_i+\ell s_i+ B_{k}^{\oplus \ell-j} \right)}  E   \right)>0,$$
for every $k\in \N$, where $B_k=\{b_1,\ldots,b_{k}\}$.
\end{theo}
To prove that these results suffice to recover Theorems \ref{dynamical reform A} and \ref{dynamical reform B}, the only necessary 
ingredient is the fact that for a general measure 
preserving system $(X,\mu,T)$ with a priori measurable pronilfactors, one can 
pass to an extension with continuous factor map onto $X$ and with topological infinite-step topological pronilfactor that agrees with that of $X$. This technical result was established in \cite[Proposition $5.7$ and Lemma $5.8$]{Kra_Moreira_Richter_Robertson:2022} for pronilfactors of finite-step. The proof extends to the case of infinite-step pronilfactors needed for our setting.

\begin{prop}\label{reduction to topological pronilfactors is possible}
Let $(X, \mu,T)$ be an ergodic system and  $a \in X$ be a transitive point. Then there exists an ergodic system $(\tilde{X}, \tilde \mu, \tilde T)$ with topological infinite-step pronilfactor, a F\o lner sequence $\Psi$ in $\N$, a point $\tilde a \in \gen(\tilde \mu, \Psi)$ and a continuous factor map $\pi \colon \tilde X \to X$ with $\pi(\tilde a) = a$. Moreover $(\tilde{X}, \tilde \mu, \tilde T)$ and $(X, \mu,T)$ are measurably isomorphic. 
\end{prop}

\begin{proof}
    By \cite[Lemma 5.7]{Kra_Moreira_Richter_Robertson:2022} we can find $(\tilde{X}, \tilde \mu, \tilde T)$ which is measurably isomorphic to $(X, \mu,T)$, $\pi \colon \tilde X \to X$ is continuous, $\pi(\tilde a) = a$ and such that the maps $\tilde \pi_k \colon \tilde X \to Z_{k}(\tilde X)$ are continuous and $Z_{k}(\tilde X)$ is isomorphic to $Z_{k}(X)$ for every $k\in \N$. It follows by the definition of infinite-step pronilfactors that
    $Z_{\infty}(\tilde X)$ is measurably isomorphic to $Z_{\infty}(X)$. By the continuity of the factor maps $\tilde{\pi}_k$ it follows that $Z_{\infty}(\tilde X)$ is a topological inverse limit of $((Z_{k}(\tilde X),T))_{k\in \N}$. Hence by the universal property of topological inverse limits, there exists a continuous factor map $\tilde \pi_{\infty} \colon \tilde X \to Z_{\infty}(\tilde X) $, concluding the proof.
\end{proof}

\begin{proof}[Proof that Theorem \ref{topological pronilfactors 1} implies 
Theorem \ref{dynamical reform B}]
Fix an ergodic system $(X,\mu,T)$, a \Folner\ sequence $\Phi=(\Phi_N)_{N\in \N}$, a point $a\in \gen(\mu,\Phi)$ and an open set $E\subset X$ with $\mu(E)>0$. By \cref{reduction to topological pronilfactors is possible} there is an ergodic extension $(\tilde{X},\tilde{\mu},\tilde{T})$ with a continuous factor map $\pi: \Tilde{X} \to X$ that has topological infinite-step pronilfactor and there is a \Folner\ sequence $\tilde{\Phi}=(\tilde{\Phi}_N)_{N\in \N}$ and a point $\tilde{a}\in \gen(\tilde{\mu},\tilde{\Phi})$ such that $\pi(\tilde{a})=a$. Let $\tilde{\sigma}$ be the measure on $\tilde{X}^{\N_0}$ given by \cref{definition of sigma} and $\sigma$ on $X^{\N_0}$ be its pushforward under $\pi^{\infty}$ (equivalently, $\sigma$ on $X^{\N_0}$ is the unique extension with projections the compatible measures $\sigma_k$ on $X^{k+1}$ defined as the pushforwards of $\tilde{\sigma}_k$ by $\pi^{k+1}$ on $\tilde{X}^{k+1}$). Finally, let $\tilde{E}=\pi^{-1}(E)$. Applying \cref{topological pronilfactors 1} to $(\tilde{X},\tilde{\mu},\tilde{T})$ we find strictly increasing 
sequences $(b_n)_{n\in \N}, (s_k)_{k\in \N}, (t_k)_{k\in \N}\subset \N$ such 
that 
$$\tilde{\sigma}_{k+1}\left(\bigcap_{i=1}^k  \prod_{j=0}^{k+1}  \bigcap_{\ell=\max(j,1)}^{i+1}  T^{-\left( t_i+\ell s_i+ B_{k}^{\oplus \ell-j}  \right)} \tilde{ E}   \right)>0,$$
for every $k\in \N$, where $B_k=\{b_1,\ldots,b_k\}$. Note that $\tilde{\sigma}_{k+1}(\{\tilde{a}\}\times X^{k+1})=1$.
By the definition of $\tilde{\sigma}$ and the fact that $\pi$ is a factor map it follows that  
\begin{equation}\label{reduction eq 1}
\sigma_{k+1}\left(\bigcap_{i=1}^k  \prod_{j=0}^{k+1}  \bigcap_{\ell=\max(j,1)}^{i+1}  T^{-\left( t_i+\ell s_i+ B_{k}^{\oplus \ell-j}  \right)} E   \right)>0,
\end{equation}
for every $k\in \N$, and also that $\sigma_{k+1}(\{a\}\times X^{k+1})=1$, for every $k\in \N$. In particular, considering only the $0$-th coordinate in \eqref{reduction eq 1} we see that 
$$a\in \bigcap_{i=1}^k    \left(\bigcap_{\ell=1}^{i+1}  T^{-\left( t_i+\ell s_i+ B_{k}^{\oplus \ell}  \right)} E \right),$$
for every $k\in \N$. It follows that, for each $k\in \N$, $1\leq i\leq k$, and $1 \leq \ell \leq i$, 
$$a \in T^{-(t_i+\ell s_i + B_k^{\oplus \ell})}  E.  $$
As this holds for every $k\in \N$ we actually get that 
$$a \in T^{-(t_i+\ell s_i + B^{\oplus \ell})}E,  $$
for all $i\in \N$ and $1\leq \ell \leq i$, where $B=(b_n)_{n\in \N}$, which is the statement of \cref{dynamical reform B}.
\end{proof}

The proof that Theorem \ref{topological pronilfactors 2} implies Theorem
\ref{dynamical reform A} is very similar and thus is omitted.

\section{\cref{topological pronilfactors 2} via progressive measures.}

\subsection{The measure $\sigma$ is progressive}\label{sigma is progressive}

In this section we recover some statements analogous to the key intermediate 
results used in the proof of the main theorem in \cite{kmrr25}. These are 
exactly the tools we   use to prove that the measure $\sigma$ defined 
above is progressive. While some 
proofs are quite similar, we unfortunately cannot always cite the 
needed results from \cite{kmrr25}, since the measures we study are defined
via the infinite-step pronilfactor, whereas the analogous measures 
therein were defined via finite-step pronilfactors. 

Throughout this section, we let $(X,\mu,T)$ be an ergodic system with topological infinite-step pronilfactor and $\sigma\in \mathcal{M}(X^{\N_0})$ be the measure defined in \cref{definition of sigma}. The first result we need is the initialization step of the inductive proofs of \cref{topological pronilfactors 2} and \cref{topological pronilfactors 1}.
\begin{prop}[Initialization step]\label{Initialization-step}
    For any $k\in \N$ and any open set $E\subset X$ with $\mu(E)>0$, there exist infinitely many $n\in \N$ such that
    $$\sigma_k(X\times T_\Delta^n(E\times \cdots \times E))>0. $$
\end{prop}
A proof of this result relies on an analog of 
\cite[Theorem $5.2$]{kmrr25} in our setting. The proof of this is 
almost identical to the one in \cite{kmrr25}, except for the
justification of exchanging the order of limits in the averages analogous to
$(5.3)$ and $(5.4)$ therein, which in our context is possible via \cref{appendix lemam to interchange limits}. However, we can also recover \cref{Initialization-step} from a new key property of our measure $\sigma$, namely that it is multiple right-progressiveness, which is introduced in \cref{sec useful lemmas}. This self-contained proof of \cref{Initialization-step} can be found in \cref{new proof of initialization-step} below.

We   use the following result which extends \cite[Lemma $6.11$]{kmrr25} for infinite-step pronilsystems. 

\begin{lemma}\label{6.11 from kmrr25}
Let $(X,\mu,T)$ be an ergodic system, $a\in \gen(\mu,\Phi)$ for some \Folner\ sequence $\Phi=(\Phi_N)_{N\in \N}$ and denote by $(Z_{\infty},m,T)$ the infinite-step pronilfactor of this system. Assume $(Y,\nu,S)$ is an infinite-step pronilsystem. Then, for every $g\in C(X)$, $y\in Y$ and $F\in C(Y)$ we have that
\begin{equation}\label{equation of 6.11}
\limsup_{N\to \infty} \left| \frac{1}{|\Phi_N|} \sum_{n\in \Phi_N} g(T^na)F(S^ny)\right| \leq \norm{\mathbb{E}(g|Z_{\infty})}_{L^1(m)} \cdot \norm{F}_{\infty}.      
\end{equation}
\end{lemma}

\begin{proof}
We find a subsequence $(\Phi_{N_j})_{j\in \N}$ so that $(a,y)$ is generic along $(\Phi_{N_j})$ for some $(T\times S)$-invariant measure $\rho$ on $X\times Y$. If $\nu$ is an $S$-invariant measure on the pronilsystem $Y$ such that $y\in \gen(\nu,\Phi)$, since $a\in \gen(\mu,\Phi)$, we see that the first and second marginals of $\rho$ are $\mu$ and $\nu$, respectively. Letting $\tilde{g}=\mathbb{E}(g | Z_{\infty})$ we only have to prove that 
\begin{equation}\label{orthogonality for infinite step pronilfactors}
\int_{X\times Y} \left(g-\tilde{g}\right) \otimes F\ \diff \rho=0    
\end{equation}
Although $\rho$ may not be ergodic, it suffices to prove \eqref{orthogonality for infinite step pronilfactors} assuming $\rho$ is an ergodic measure and then using its ergodic decomposition.

Using \cref{infinite step nilfactor and HK seminorms} we have that 
$\norm{(g-\tilde{g}) \otimes \mathbbm{1}}_{U^k(X\times Y,\rho,T\times S)}=0$ for all $k\in \N$. By the same proposition we then see that $(g-\tilde{g}) \otimes \mathbbm{1}$ is orthogonal to the infinite-step pronilfactor of $(X\times Y,\rho,T\times S)$. On the other hand, $(Y,\nu,S)$ is a factor of $(X\times Y,\rho,T\times S)$, and so $\mathbbm{1} \otimes F$ is measurable with respect to the maximal factor of $(X\times Y,\rho,T\times S)$ that is an infinite-step pronilsystem, namely its infinite-step pronilfactor. The functions $(g-\tilde{g}) \otimes \mathbbm{1}$ and $\mathbbm{1} \otimes F$ are thus orthogonal in $L^2(\rho)$  and \eqref{orthogonality for infinite step pronilfactors} follows.
\end{proof}

Because each $\sigma_k$, for $k\geq 2$, is invariant under 
$Id\times T \times \dots \times T^k$ and its marginals satisfy 
${\sigma_k}_i \leq i \mu$, for $i=1,\ldots,k$, we can apply 
\cite[Theorem $6.6$]{kmrr25} directly. 

\begin{prop} \label{theorem 6.6 from kmrr}
Let $k\in \N$, with $k\geq 2$ and let $\Phi=(\Phi_N)_{N\in \N}$ be \Folner\ sequence in $\N$. Then, for any $f_1,\ldots,f_{k-1}\in L^{\infty}(\mu)$ and any bounded sequence $b:\N\to \C$ we have that 
\begin{multline*}
\limsup_{N\to \infty} \norm{\frac{1}{|\Phi_N|}\sum_{n\in \Phi_N} b(n) \cdot \left( \1 \otimes T^nf_1 \otimes \dots \otimes T^nf_{k-1} \otimes \1 \right)}_{L^2(\sigma_k)} \\
\leq C_k \cdot \norm{b}_{\infty} \cdot \min\{\norm{f_i}_{U^{k}(X,\mu,T)}: 1\leq i \leq k-1\},    
\end{multline*}
where $C_k>0$ is a constant that only depends on $k$.
\end{prop}

This, in combination with \cref{6.11 from kmrr25}, together imply the analog of \cite[Theorem $6.2$]{kmrr25} for $\sigma_k$, by repeating the proof in \cite[Section $6.4$]{kmrr25} (we also use \cref{appendix lemam to interchange limits} that allows us to interchange the order of analogous limits). Namely, we have the following. 

\begin{prop}\label{analogue of 6.2}
Let $k\in \N$ with $k\geq 2$. For any function $F\in C(X^k)$ we have that
$$ \lim_{N\to \infty} \norm{\frac{1}{|\Phi_N|}\sum_{n\in \Phi_N} T_{\Delta}^n(F\otimes \1) - \frac{1}{|\Phi_N|}\sum_{n\in \Phi_N} T_{\Delta}^n(\1\otimes F) }_{L^2(\sigma_k)} =0.$$
\end{prop}

Finally, \cite[Theorem $6.1$]{kmrr25} also applies for each measure 
$\sigma_k$, $k\in \N$, (more precisely, for the projection of $\sigma_k$ to
its last $k$-coordinates) because the only assumption there is that the 
measure is $(T\times T^2 \times \dots \times T^k)-$invariant. 

\begin{prop}\label{recurrence-prop-k-2}
    Let $k\in \N$ with $k\geq 2$. Let $F\in C(X^k)$ be a nonnegative function and $\Phi=(\Phi_N)_{N\in \N}$ be a \Folner\ sequence in $\N$. Then
    $$\int \1 \otimes F d\sigma_k>0 \Longrightarrow \liminf_{N\to\infty}\frac{1}{|\Phi_N|} \sum_{n\in \Phi_N} \int T_\Delta^n(\1 \otimes F) \cdot (\1\otimes F) \diff\sigma_k>0.$$ 
\end{prop}
\begin{remark}\label{quantitative recurrence-prop-k-2}
In fact, it follows from the proof of \cite[Theorem $6.1$]{kmrr25}, and in particular from the use of Furstenberg's 
multiple recurrence theorem, that in the notation of the proposition, for any $\epsilon>0$ there is $\delta>0$ so that 
$$\int \1 \otimes F d\sigma_k> \epsilon \Longrightarrow \liminf_{N\to\infty}\frac{1}{|\Phi_N|} \sum_{n\in \Phi_N} \int T_\Delta^n(\1 \otimes F) \cdot (\1\otimes F) \diff\sigma_k>\delta.$$
\end{remark}
As shown in \cite[Section $6.1$]{kmrr25}, these last two results together 
imply that each $\sigma_k$, $k\geq 2$, is left-progressive (see also \cref{recurrence-prop-k}) according to \cref{definition left-right progressive measure}. Since $P_k\sigma=\sigma_k$, we have shown the following. 

\begin{prop}\label{left-progressive proof}
The measure $\sigma$ defined via \eqref{eq definition sigma} is 
left-progressive.    
\end{prop}

As a concluding remark, we mention that in fact (this is implicit in 
\cite{kmrr25}) the proof of \cref{left-progressive proof} actually gives the following result (which we do not need to use here).

\begin{prop}\label{recurrence-prop-k}
    Let $F\in C(X^k)$ be a nonnegative function with $\int \1 \otimes F \diff \sigma_k>0$. Then, 
    $$\liminf_{N\to\infty} \frac{1}{|\Phi_N|} \sum_{n\in \Phi_N} \int (T_\Delta^n F(a,x_1,\ldots,x_{k-1}) \otimes \1) \cdot (\1 \otimes F(x_1,\ldots,x_k)) \diff \sigma_k(x)>0.$$
\end{prop}

Next we show that $\sigma$ is right-progressive.

\begin{prop}\label{right-progressive proof}
The measure $\sigma$ defined via \eqref{eq definition sigma} is 
right-progressive.    
\end{prop}

\begin{proof}
Let $k\in \N$ with $k\geq 2$ be fixed. Let $U_1,\ldots,U_{k-1}\subset X$ be open sets with 
$$\sigma_k(X \times U_1 \times U_2 \times \dots \times U_{k-1}\times X)>0.$$
As $V=U_1\times \dots \times U_{k-1} \subset X^{k-1}$ is an open set, we can find a function $G\in C(X^{k-1})$ such that $0\leq G \leq \mathbbm{1}_{V}$ and $\int_{X^{k+1}} \1 \otimes G \otimes \mathbbm{1} \diff \sigma_k>0.$ We thus have 
\begin{multline*}
\liminf_{N\to \infty} \frac{1}{|\Phi_N|} \sum_{n\in \Phi_N} \sigma_k(( X \times V \times X) \cap T_{\Delta}^{-n}(X \times X \times V)) \\ \geq \liminf_{N\to \infty} \frac{1}{|\Phi_N|} \sum_{n\in \Phi_N} \int_{X^{k+1}} (\1 \otimes G \otimes \mathbbm{1})\cdot T_{\Delta}^n (\1\otimes \mathbbm{1} \otimes G)  \diff \sigma_k \\
=\liminf_{N\to \infty} \frac{1}{|\Phi_N|} \sum_{n\in \Phi_N} \int_{X^{k+1}} (\1 \otimes G \otimes \1)\cdot T_{\Delta}^n (\1 \otimes G \otimes \mathbbm{1})  \diff \sigma_k,
\end{multline*}
where the equality follows by an application of \cref{analogue of 6.2} for the continuous function $\1 \otimes G \in C(X^{k})$ and the measure $\sigma_k$. Applying \cref{recurrence-prop-k-2} for the continuous function $\1 \otimes G \otimes \mathbbm{1} \in C(X^{k+1})$ it follows that the last expression is positive. Hence, there exist infinitely many $n\in \N$ such that 
$$\sigma_k((X \times U_1 \times \dots \times U_{k-1} \times X) \cap T_{\Delta}^{-n}(X \times X \times U_1 \times \cdots \times U_{k-1}))>0$$
and this concludes the proof.
\end{proof}

We remark in passing that the proof of \cref{right-progressive proof}
actually shows the following result, although we do not need to use it. 

\begin{prop}\label{recurrence-prop-k'}
    Let $F\in C(X^{k-1})$ be a nonnegative function with $\int \1 \otimes F \otimes \1 \diff \sigma_{k}>0$. Then, 
    $$\liminf_{N\to\infty} \frac{1}{|\Phi_N|} \sum_{n\in \Phi_N} \int \1 \otimes \1 \otimes (T_\Delta^n F(x_2,\ldots,x_{k})) \cdot (\1 \otimes F(x_1,\ldots,x_{k-1}) \otimes \1) \diff \sigma_{k}(x)>0.$$
\end{prop}

As a corollary of Propositions \ref{left-progressive proof} and \ref{right-progressive proof} we deduce the following.

\begin{theo}\label{progressive proof}
The measure $\sigma$ defined via \eqref{eq definition sigma} is 
progressive.   
\end{theo}
\subsection{Proof of \cref{topological pronilfactors 2}}

We   now prove \cref{topological pronilfactors 2} using the results 
of the previous subsection and an inductive argument. We repeat the 
statement of the theorem for the convenience of the reader.

\begin{named}{\cref{topological pronilfactors 2}}{}Let $(X,\mu,T)$ be an ergodic system with topological infinite-step 
pronilfactor. Let $\Phi=(\Phi_N)_{N\in \N}$ be a \Folner\ sequence, let 
$a\in \gen(\mu,\Phi)$ and let $\sigma$ be the measure defined in \cref{definition of sigma}. Then, for $\ell\in \N$ fixed and
any open set $E\subset X$ with $\mu(E)>0$ there exist 
sequences $(b_n)_{n\in \N}, (t_k)_{k\in \N}\subset \N$ such 
\begin{equation*}
       0<\sigma_{k+\ell+1}\left( \bigcap_{i=1}^{k+1}  \prod_{j=0}^{k+\ell+1}\bigcap_{h=\max(j,i)}^{i+\ell}  T^{-\left( t_i+ B_k^{\oplus h-j} \right)}  E   \right),  
   \end{equation*}
for every $k\in \N$, where $B_k=\{b_1,\ldots,b_{k}\}$.
\end{named}

The diagram sketching the proof \cref{topological pronilfactors 2} is the following.

\begin{figure}[h!]
\centering
\resizebox{1\textwidth}{!}{%
\begin{tikzpicture}[
    font=\LARGE,
    >=Stealth,
    node distance=2cm and 3cm,
    box/.style={rectangle, draw, minimum width=4cm, minimum height=1.5cm, align=center}
]
\node[box] (A) {Initialization step};
\node[box, below=of A] (B) {Right-progressiveness};
\node[box, below=of B] (C) {Left-progressiveness};

\draw[->] (A) -- (B);
\draw[->] (B.east) -- ++(2cm,0) |- (C.east);
\draw[->] (C.west) -- ++(-2cm,0) |- (B.west);

\draw[->, dashed] (A.east) -- ++(4cm,0) node[right]{$\begin{matrix}
    \text{Finds } t_1  \text{ such that} \\
 \sigma_{\ell+2}(X\times T^{-t_1}E \times\cdots \times T^{-t_1}E)>0
\end{matrix} $};
\draw[->, dashed] (B.east) ++(2cm,-2cm) -- ++(2cm,0) node[right]{$\begin{matrix}
    \text{Finds } b_k \text{ and } t_k  \text{ such that} \\
  \sigma_{k+\ell+1}\left(\bigcap_{i=1}^{k+1}  \prod_{j=0}^{k+\ell+1}\bigcap_{h=\max(j,i)}^{i+\ell}  T^{-\left( t_i+ B_k^{\oplus h-j} \right)}  E   \right)>0.
\end{matrix} $};
\end{tikzpicture}}
\caption{Diagram of the proof of \cref{topological pronilfactors 2}}
\label{fig:2}
\end{figure}

\begin{proof}
    By \cref{Initialization-step} we can find $t_1\in \N$ such that 
    $$\sigma_{\ell+1}(X\times T^{-t_1}E \times\cdots \times T^{-t_1}E)>0. $$
    As $\sigma$ is right-progressive we can find $t_2'$ such that if $t_2=t_2'+t_1$, it holds that 
    $$\sigma_{\ell+2}( (X\times T^{-t_1}E \times\cdots \times T^{-t_1}E \times X)\cap (X\times X\times T^{-t_2}E \times\cdots \times T^{-t_2}E)) >0. $$
    Using left-progressiveness we obtain $b_1\in \N$ such that the $\sigma_{\ell+2}$ measure of  
    \begin{align*}
       &\Big(T^{-(t_1+b_1)}E \times \left( T^{-(t_1+b_1)}E \cap T^{-t_1}E \right) \times \cdots \times \left( T^{-(t_1+b_1)}E\cap T^{-t_1}E \right) \times T^{-t_1}E \times X \Big) \bigcap \\
       & \Big( X\times T^{-(t_2+b_1)}E \times \left( T^{-(t_2+b_1)} \cap T^{-t_2}E \right) \times \cdots \times \left( T^{-(t_2+b_1)}E\cap T^{-t_2}E \right) \times T^{-t_2}E \Big) 
    \end{align*}
    is positive. Recall the notational conventions from \cref{notational conventions} and observe that this can be rewritten as 
    $$\sigma_{\ell+2}\left(\bigcap_{i=1}^{2}  \prod_{j=0}^{\ell+2}\bigcap_{h=\max(j,i)}^{i+\ell}  T^{-\left( t_i+ B_1^{\oplus h-j} \right)}  E   \right)>0,$$
    where $B_1=\{b_1\}$, using for example that $B_1^{\oplus h}=\emptyset$ and so $T^{-(t_2+B_1^{\oplus h})}E=X$, for $h\geq 2$. In this fashion we   construct the sequences of the statement inductively.

    Assume we have constructed $b_1<\cdots<b_k$ and $t_1<\cdots< t_{k+1}\in \N$ such that if $B_k=\{b_1,\ldots,b_k\}$, we have that
    \begin{equation}\label{eq-measure-conj-2}
        \sigma_{k+\ell+1}\left(\bigcap_{i=1}^{k+1}  \prod_{j=0}^{k+\ell+1}\bigcap_{h=\max(j,i)}^{i+\ell}  T^{-\left( t_i+ B_k^{\oplus h-j} \right)}  E   \right)>0.
    \end{equation}
    We note in particular that 
    \begin{equation}\label{eq-measure-conj-2 (2)}
    \sigma_{k+\ell+1}\left( \prod_{j=0}^{k+\ell+1}\bigcap_{h=\max(j,k+1)}^{k+1+\ell}  T^{-\left( t_{k+1}+ B_k^{\oplus h-j} \right)}  E   \right)>0.    
    \end{equation}
    As $\sigma$ is a right-progressive measure we can apply the defining property in \cref{definition left-right progressive measure}, only considering the set isolated in \eqref{eq-measure-conj-2 (2)}, in order to find $t_{k+2}'$ such that if $t_{k+2}=t_{k+2}'+t_{k+1}$, then the $\sigma_{k+\ell+2}$ measure of the set 
    \begin{align*}
           &\left[\left(\bigcap_{i=1}^{k+1}  \prod_{j=0}^{k+\ell+2}\bigcap_{h=\max(j,i)}^{i+\ell}  T^{-\left( t_i+ B_k^{\oplus h-j} \right)}  E   \right)\right] \cap  \left[ X\times  \prod_{j=0}^{k+\ell+1}  \left(\bigcap_{h=\max(j,k+1)}^{k+1+\ell}  T^{-\left( t_{k+2}+ B_k^{\oplus h-j} \right)}  E \right) \right]\\
         \end{align*}
         is positive. The projection onto the $0$-th coordinate of this set equals 
         $$ \bigcap_{i=1}^{k+1} \bigcap_{h=i}^{i+\ell}  T^{-\left( t_i+ B_k^{\oplus h}\right)} =  \bigcap_{i=1}^{k+2} \bigcap_{h=i}^{i+\ell}  T^{-\left( t_i+ B_k^{\oplus h}\right)} $$ 
         and the projection onto the rest $k+\ell+2$ coordinates becomes 
         \begin{align*}
            & \bigcap_{i=1}^{k+1}  \prod_{j=1}^{k+\ell+2} \left(\bigcap_{h=\max(j,i)}^{i+\ell}  T^{-\left( t_i+ B_k^{\oplus h-j} \right)}  E  \cap \bigcap_{h=\max(j-1,k+1)}^{k+1+\ell} T^{-\left(t_{k+2}+ B_k^{\oplus h-j+1}\right)} E\right) \\
           &=  \bigcap_{i=1}^{k+1}  \prod_{j=1}^{k+\ell+2}\left(\bigcap_{h=\max(j,i)}^{i+\ell}  T^{-\left( t_i+ B_k^{\oplus h-j} \right)}  E  \cap \bigcap_{h=\max(j,k+2)}^{k+2+\ell} T^{-\left(t_{k+2}+ B_k^{\oplus h-j}\right)}E \right)\\
           &=\bigcap_{i=1}^{k+2}  \prod_{j=1}^{k+\ell+2}\bigcap_{h=\max(j,i)}^{i+\ell}  T^{-\left( t_i+ B_k^{\oplus h-j} \right)}  E. 
           \end{align*}
    In total, this gives that 
    $$\sigma_{k+\ell+2}\left(\bigcap_{i=1}^{k+2}  \prod_{j=0}^{k+\ell+2}\bigcap_{h=\max(j,i)}^{i+\ell}  T^{-\left( t_i+ B_k^{\oplus h-j} \right)}  E   \right)>0. $$
   By left-progressiveness, we obtain $b_{k+1}>b_k$ such that the $\sigma_{k+\ell+2}$ measure of
   \begin{align*}
      &\left(\bigcap_{i=1}^{k+2}  \prod_{j=0}^{k+\ell+2}\bigcap_{h=\max(j,i)}^{i+\ell}  T^{-\left( t_i+ B_k^{\oplus h-j} \right)}  E   \right) \cap \left( \bigcap_{i=1}^{k+2}  \prod_{j=0}^{k+\ell+2}\bigcap_{h=\max(j+1,i)}^{i+\ell}  T^{-\left( t_i+ b_{k+1}+ B_k^{\oplus h-j-1} \right)}  E  \right) \\
      &=\bigcap_{i=1}^{k+2} \prod_{j=0}^{k+\ell+2} \left(\bigcap_{h=\max(j,i)}^{i+\ell}  T^{-\left( t_i+ B_k^{\oplus h-j} \right)}  E    \cap \bigcap_{h=\max(j+1,i)}^{i+\ell}  T^{-\left( t_i+b_{k+1}+ B_k^{\oplus h-j-1} \right)}  E \right)
   \end{align*}  
    is positive. If we let $B_{k+1}=\{b_1,\ldots,b_k,b_{k+1}\}$, we have thus established that 
   \begin{equation*}
       0<\sigma_{k+\ell+2}\left( \bigcap_{i=1}^{k+2}  \prod_{j=0}^{k+\ell+2}\bigcap_{h=\max(j,i)}^{i+\ell}  T^{-\left( t_i+ B_{k+1}^{\oplus h-j} \right)}  E   \right),  
   \end{equation*}
   concluding the induction. 
\end{proof}

\section{\cref{topological pronilfactors 1} via multiple right-progressive measures} \label{sec mult recurrence}

\subsection{The measure $\sigma$ is multiple right-progressive} \label{sec useful lemmas}

For a metric space $X$, a homeomorphism $T \colon X \to X$ and a vector $u = (u_1, \ldots, u_k) \in \N^k$ we write $T^u = T^{u_1} \times \cdots \times T^{u_k}$.  Based on Furstenberg-Katznelson's multiple recurrence theorem \cite[Theorem A]{Furstenberg_Katznelson78}, we extend \cite[Theorem 6.3]{kmrr25} to a multiple recurrence statement.

\begin{theo}\label{T-Delta-Furstenberg}
    Let $X$ be a compact metric space, let $T_1, \ldots, T_{\ell} \colon X \to X$ be commuting homeomorphisms for some $\ell \in \N$ and $\Phi = (\Phi_N)_{N \in \N}$ be a F\o lner sequence. Let also $k \in \N$, $v_1, \ldots, v_{\ell} \in \N^k$ and $\nu \in \cM(X^k)$ be a Borel probability measure invariant under $T_j^{v_j},$ for $j=1,\ldots,\ell$. For every measurable set $A = A_1 \times \cdots \times A_k \subset X^k$ with $\nu(A) >0$, and every $u_1, \ldots, u_{\ell} \in \N^k$, it holds that
    \begin{equation}
        \liminf_{N \to \infty} \frac{1}{|\Phi_N|} \sum_{n \in \Phi_N} \nu(A \cap (T^{u_1}_1)^{-n} A \cap \cdots \cap (T^{u_{\ell}}_{\ell})^{-n} A) >0 
    \end{equation}
\end{theo}
\begin{remark*}
    In this work we are merely concerned with the case $T_1 = \cdots = T_{\ell}  = T$ for a fixed $T \colon X \to X$, $v_i = (1, 2, \ldots, k)$ and $u_i = (i, \ldots, i) \in \N^k$ for all $i=1, \ldots, \ell$. But we believe the general statement is of independent interest. Notice that when $v_i = u_i$ for all $i=1, \ldots, \ell$, the statement is a special case of Furstenberg-Katznelson's multiple recurrence theorem \cite{Furstenberg_Katznelson78}. We also remark that the assumption that $A$ is a product set is necessary, as shown in \cite[Example $6.4$]{kmrr25} for the single recurrence statement proven there.
\end{remark*}

\begin{proof}
    This proof is similar to the one of \cite[Theorem 6.3]{kmrr25}. We let $v_i=(v_{i,1},\ldots, v_{i,k})\in \N^k$ and $u_i =  (u_{i,1} , \ldots, u_{i,k})\in \N^k$, for $i=1,\ldots,\ell$, be vectors as in the statement of the theorem. 
    
    Let $c= LCM(v_{i,j} \colon 1 \leq i \leq \ell, 1 \leq j \leq k)$ be the least common multiple of the coordinates of $v_1, \ldots, v_{\ell}$ and set $w_{i,j}= c \ \dfrac{u_{i,j}}{v_{i,j}} \in \N$, for $1\leq i\leq \ell$ and $1\leq j \leq k$. Observe that
\begin{equation*}
    (T_i^{v_i})^{-w_{i,j}}(X^{j-1} \times A_j \times X^{k-j}) = (T^{u_i}_i)^{-c}(X^{j-1} \times A_j \times X^{k-j}),
\end{equation*}
for all $i = 1, \ldots, \ell$ and $j=1, \ldots, k$. Thus, for any $n\in \N$ we see that 
\begin{align*}
   (T^{u_i}_i)^{-cn}A &= \bigcap_{j=1}^k (T^{u_i}_i)^{-cn}(X^{j-1} \times A_j \times X^{k-j}) \\
   &= \bigcap_{j=1}^k (T_i^{v_i})^{-w_{i,j}n}(X^{j-1} \times A_j \times X^{k-j}) \supset \bigcap_{j=1}^k (T_i^{v_i})^{-w_{i,j}n} A. 
\end{align*}
For the \Folner\ sequence $\Psi=(\Psi_N)_{N\in \N}$ given by $\Psi_N= \Phi_N/c = \{ n \in \N \colon cn \in \Phi_N \}$, we have 
\begin{align*}
    \liminf_{N \to \infty}& \frac{1}{|\Phi_N|} \sum_{n \in \Phi_N} \nu(A \cap (T^{u_1}_1)^{-n} A \cap \cdots \cap (T^{u_{\ell}}_{\ell})^{-n} A) \\
    &\geq \liminf_{N \to \infty} \frac{1}{c|\Psi_N|} \sum_{n \in \Psi_N} \nu(A \cap (T^{u_1}_1)^{-cn} A \cap \cdots \cap (T^{u_{\ell}}_{\ell})^{-cn} A) \\
    &\geq \liminf_{N \to \infty} \frac{1}{c|\Psi_N|} \sum_{n \in \Psi_N} \nu \Big(A \cap \Big(\bigcap_{j=1}^k (T_1^{v_1})^{-w_{1,j}n} A \Big)  \cap \cdots \cap \Big(\bigcap_{j=1}^k (T_{\ell}^{v_{\ell}})^{-w_{{\ell},j}n} A \Big) \Big).
\end{align*}
Since $\nu$ is invariant under $T_1^{v_1}, \ldots, T_{\ell}^{v_{\ell}}$ (and therefore by $(T_i^{v_i})^{w_{i,j}}$ for all $i = 1 , \ldots, \ell$ and $j = 1, \ldots, k $), the last average is positive by Furstenberg-Katznelson's multiple recurrence theorem, so we conclude. 
\end{proof}

\begin{remark*}
    Notice that if we consider the special case where $T_1 = \cdots = T_{\ell}  = T$ for a fixed $T \colon X \to X$ in the previous proof, it suffices to apply Furstenberg's multiple recurrence. 
\end{remark*}

Throughout this section, we let $(X,\mu,T)$ be an ergodic system with topological infinite-step pronilfactor $(Z_{\infty},m,T)$, and $\sigma\in \mathcal{M}(X^{\N_0})$ be the measure defined in \cref{definition of sigma} for some generic point $a\in \gen(\mu,\Phi)$.
We use \cref{T-Delta-Furstenberg} to obtain the following result. 
\begin{cor}\label{T-Delta-Furstenberg-cor}
       Let $k,\ell \in \N$ with $k\geq 2$. Let $f_1, \ldots, f_k  \in C(X)$ be nonnegative functions and $\Phi=(\Phi_N)_{N\in \N}$ be a \Folner\ sequence in $\N$. If $F=\1\otimes f_1\otimes\cdots \otimes f_k$ and
    $$\int F\diff \sigma_k>0, $$ 
   then we have that 
    $$  \liminf_{M\to\infty}\liminf_{N\to\infty}\frac{1}{|\Phi_M|} \sum_{m\in \Phi_M}  \frac{1}{|\Phi_N|} \sum_{n\in \Phi_N} \int F \cdot \prod_{i=1}^{\ell}  T_{\Delta}^{n+im} F \diff\sigma_{k}>0.$$  
\end{cor}
\begin{proof}
    Denote $F^*=f_1\otimes \cdots \otimes f_k$. Letting $\supp(f_i)=\overline{\{x: f_i(x)>0\}}$, for $i=1,\ldots,k$, and setting $A=\supp(f_1)\times \cdots\times \supp(f_k)$, it follows by the assumption that $\int F^* \diff P_k^*\sigma>0$ and hence that $P_k^*\sigma(A)>0$. This allows us to apply \cref{T-Delta-Furstenberg} for the set $A\subset X^k$ defined above and the measure $P_k^*\sigma$, which is invariant under $T\times T^2 \times \cdots \times T^k$. In particular, with $k,\ell \in \N$ as above and $u_i=(i,\ldots,i)$ for each $i\in [\ell]$, and $T_1=\cdots=T_{\ell}=T$, we obtain 
    $$\liminf_{M\to \infty} \frac{1}{|\Phi_M|}\sum_{m\in \Phi_M} P_k^*\sigma(A \cap T_\Delta^{-m}A \cap \cdots \cap T_\Delta^{-\ell m}A)>0. $$
    As $f_1,\ldots,f_k$ are continuous functions, this implies that
     $$\liminf_{M\to \infty} \frac{1}{|\Phi_M|}\sum_{m\in \Phi_M} \int F^* \cdot \prod_{i=1}^\ell T_\Delta^{im}F^*  \diff P_k^*\sigma>0. $$
 Equivalently, for some $\epsilon>0$, the set 
    $$P_{\epsilon}=\{m\in \N :  \int F^* \cdot \prod_{i=1}^\ell T_\Delta^{im}F^*  \diff P_k^*\sigma>\epsilon\} $$
   is such that $\underline{\diff}_{\Phi}(P_{\epsilon})>0$. Then, it follows by \cref{recurrence-prop-k-2} (see also Remark \ref{quantitative recurrence-prop-k-2}) that there exists $\epsilon'>0$ such that for each $m\in P_{\epsilon}$,
   $$  \liminf_{N\to\infty} \frac{1}{|\Phi_N|} \sum_{n\in \Phi_N} \int F^* \cdot \prod_{i=1}^{\ell}  T_{\Delta}^{im} F^* \cdot T_{\Delta}^n F^* \cdot \prod_{i=1}^{\ell}  T_{\Delta}^{n+im} F^* \diff P_k^*\sigma>\epsilon',$$
   which in turn implies that
    $$  \liminf_{N\to\infty} \frac{1}{|\Phi_N|} \sum_{n\in \Phi_N} \int F^* \cdot \prod_{i=1}^{\ell}  T_{\Delta}^{n+im} F^* \diff P_k^*\sigma>\epsilon'',$$
    for some $\epsilon''>0$. As $\underline{\diff}_{\Phi}(P_\epsilon)>0$ we conclude that
       $$  \liminf_{M\to\infty}\frac{1}{|\Phi_M|} \sum_{m\in \Phi_M} \liminf_{N\to\infty}  \frac{1}{|\Phi_N|} \sum_{n\in \Phi_N} \int F^* \cdot \prod_{i=1}^{\ell}  T_{\Delta}^{n+im} F^* \diff P_k^*\sigma>0,$$
       which implies the conclusion of the theorem.
\end{proof}

The main result of this section is the following proposition which is an 
indirect consequence of \cref{T-Delta-Furstenberg}. 

\begin{prop}\label{multiple-recurrence-T-Delta-uncoupled}
    Let $k,\ell\in \N$ with $k\geq 2$. Let $f_1, \ldots, f_k  \in C(X)$ be nonnegative functions and $\Phi=(\Phi_N)_{N\in \N}$ be a \Folner\ sequence in $\N$. If $F=\1\otimes f_1\otimes\cdots \otimes f_k$ and
    $$\int Fd\sigma_k>0, $$
    then we have that 
    $$  \liminf_{M\to\infty}\liminf_{N\to\infty}\frac{1}{|\Phi_M|} \sum_{m\in \Phi_M}  \frac{1}{|\Phi_N|} \sum_{n\in \Phi_N} \int (F\otimes \1^{\otimes \ell}) \cdot \prod_{i=1}^{\ell}  (T_{\Delta}^{n+im} (\1^{\otimes i} \otimes F\otimes \1^{\otimes \ell -i})) \diff\sigma_{k+\ell}>0.$$ 
\end{prop}

Observe that for $\ell=1$ this implies \cref{recurrence-prop-k'}. This result 
allows us to utilize the following property that we naturally term multiple 
right-progressiveness.

\begin{defn} \label{def multiple right progressive}
    Let $(X,T)$ be a topological system. We say that a Borel probability measure $\tau\in \mathcal{M}(X^{\N_0})$ is \textit{multiple right-progressive} if for each $k,\ell \in \N$ and any open sets $U_1,\ldots,U_k\subset X$ with 
    $$\tau_k(X\times U_1\times \cdots \times U_k)>0, $$
    there exist infinitely many $n,m\in \N$ such that 
    $$\tau_{k+\ell}\left((X\times U \times X^{\ell})\cap T_\Delta^{-(n+m)}(X^2 \times U\times X^{\ell-1}) \cap \cdots \cap T_\Delta^{-(n+\ell m)}(X^{\ell+1}\times U)\right)>0,$$
    where $U=U_1\times \cdots \times U_k$ and $\tau_k$ and $\tau_{k+\ell}$ are the projections onto the first $(k+1)$ and $(k+\ell+1)$ coordinates of $\tau$, respectively.
\end{defn}
\begin{remark*}
We want to emphasize that \cref{multiple-recurrence-T-Delta-uncoupled} 
implies that $\sigma$ is multiple right-progressive. Moreover, note 
that for $\ell=1$, this property coincides with right-progressiveness. 
\end{remark*}
\begin{remark}\label{new proof of initialization-step}
    The initialization step, i.e. \cref{Initialization-step}, for the proofs of \cref{topological pronilfactors 2} and \cref{topological pronilfactors 1}, can be obtained as a consequence of multiple right-progressiveness and left-progressiveness. Indeed, by multiple right-progressiveness for the case $k=1$, $\ell \in \N$ and $U_1=E$ an open set with $\mu(E)>0$, we obtain infinitely many $n,m\in \N$ such that
    $$\sigma_{\ell +1}(X \times E\times T^{-(n+m)}E \times  T^{-(n+2m)}E \times \cdots \times  T^{-(n+\ell m)}E)>0.   $$
    Disregarding the second coordinate and using left progressiveness, we find $\tilde{n}$ such that
    \begin{align*}
        0&<\sigma_{\ell }(X \times T^{-(\tilde{n}+m)}E \times  T^{-(\tilde{n}+2m)}E \times \cdots \times  T^{-(\tilde{n}+\ell m)}E)\\
        &= \sigma_{\ell}((I_d\times \tilde{T})^{-m}(X \times T^{-\tilde{n}}E \times  T^{-\tilde{n}}E \times \cdots \times  T^{-\tilde{n}}E))\\
        &= \sigma_{\ell}(X\times T_\Delta^{-\tilde{n}} (E\times \dots \times E)),
    \end{align*}
    obtaining the conclusion of \cref{Initialization-step}. 
\end{remark}
The proof of \cref{multiple-recurrence-T-Delta-uncoupled} for infinite-step pronilsystems is straightforward from the definition of $\sigma$. In fact, in this special case we do not need to average along the transformation $T_\Delta^n$, and the result is a corollary of \cref{T-Delta-Furstenberg}. Our goal then is to show that the infinite-step pronilfactor is characteristic. To prove this, it would be enough to show the related ergodic averages are controlled by uniformity seminorms. As we show in \cref{subsec ex double average} these seminorms do not always control the single average, hence the need to introduce a second averaging over $T_\Delta^n$. However, for the double averages in \cref{multiple-recurrence-T-Delta-uncoupled} we have the following.

\begin{lemma}\label{projection-lemma-for-qmri}
Let $(X, \mu,T)$ be an ergodic system, $k,\ell \in \N$ with $k\geq 2$, let $\tau\in \cM(X^{k})$ be a $(T\times T^2 \times \cdots \times T^k)$-invariant measure with marginals $\tau_i \leq C \mu $ for some $C>0$. Moreover, let $f_1, \ldots, f_k  \in C(X)$ and $\Phi=(\Phi_N)_{N \in \N}$ a \Folner\ sequence in $\N$. If $F= f_1 \otimes  \cdots \otimes f_k $, then
\begin{equation*}
    \limsup_{M \to \infty} \limsup_{N \to \infty} \norm{ \frac{1}{|\Phi_M|} \sum_{m \in \Phi_M} \frac{1}{|\Phi_N|} \sum_{n \in \Phi_N} \prod_{i=1}^{\ell} T_\Delta^{n+im}  F }_{L^2(\tau)} \leq C_{k,\ell} \min_{1\leq j\leq k}\{\norm{f_j}_{U^{k \cdot \ell}(X,\mu,T)} \},
\end{equation*} 
where $ C_{k,\ell}>0$ is a constant that only depends on $k$ and $\ell$.
\end{lemma}

\begin{proof}
    Let $c = LCM(2, \ldots, k)$. Considering the F\o lner sequence $\Psi=(\Psi_N)_{N\in \N}$ given by 
\begin{equation*}
    \Psi_N = \Phi_N/c = \{ n : cn \in \Phi_N\},\ \text{for}\ N\in \N, 
\end{equation*}
and applying \cite[Lemma $6.8$]{kmrr25}, we get that 
\begin{align*}
     & \limsup_{M \to \infty} \limsup_{N \to \infty} \norm{ \frac{1}{|\Phi_M|} \sum_{m \in \Phi_M} \frac{1}{|\Phi_N|} \sum_{n \in \Phi_N} \prod_{i=1}^{\ell} T_\Delta^{n+im} F }_{L^2(\tau)}  \\
     &\leq \frac{1}{c^2} \sum_{r,r'=0}^{c-1} \left( \limsup_{M \to \infty} \limsup_{N \to \infty}  \norm{ \frac{1}{|\Psi_M|} \sum_{m \in \Psi_M} \frac{1}{|\Psi_N|} \sum_{n \in \Psi_N} \left( \prod_{j=1}^{k}\prod_{i=1}^{\ell} (T^{cn+icm+r+ir'} f_{j})^{[j]} \right) }_{L^2(\tau)}  \right),
\end{align*}
where for a function $f \in C(X)$ we denote $f^{[j]}(x_1, \ldots, x_k) = f(x_j)$, for $(x_1, \ldots, x_k) \in X^k$. Fixing $r,r' \in \{0, \ldots , c-1\}$ and letting $S= T \times T^2 \times \cdots \times T^k$,  and $c_j = c/j$ for $j=1, \ldots, k$, we get that 
\begin{align*}
    &\limsup_{M \to \infty} \limsup_{N \to \infty}  \norm{ \frac{1}{|\Psi_M|} \sum_{m \in \Psi_M} \frac{1}{|\Psi_N|} \sum_{n \in \Psi_N} \left( \prod_{j=1}^{k}\prod_{i=1}^{\ell}  (T^{cn+icm+r+ir'} f_{j})^{[j]} \right) }_{L^2(\tau)}  \\
    = & \limsup_{M \to \infty} \limsup_{N \to \infty}  \norm{ \frac{1}{|\Psi_M|} \sum_{m \in \Psi_M} \frac{1}{|\Psi_N|} \sum_{n \in \Psi_N} \left( \prod_{j=1}^{k} S^{c_jn} \left( \prod_{i=1}^{\ell} S^{ic_jm} (T^{r+ir'}f_{j})^{[j]} \right) \right) }_{L^2(\tau)}. 
\end{align*}
By \cite[Theorem 1.1]{Zorin-Kranich16}, the previous double limsup is actually a limit, and it is equal to  
\begin{equation}\label{eq 1 in seminorm control proof}
\lim_{N\to \infty}\lim_{M\to \infty} \norm{ \frac{1}{|\Psi_M|} \sum_{m \in \Psi_M} \frac{1}{|\Psi_N|} \sum_{n \in \Psi_N} \left( \prod_{j=1}^{k} S^{c_jn} \left( \prod_{i=1}^{\ell} S^{ic_jm} (T^{r+ir'}f_{j})^{[j]} \right) \right) }_{L^2(\tau)}.  
\end{equation}
Notice that for an integer $u>k c$, the function $(i,j)\in \{1, \ldots, \ell\}\times \{1, \ldots, k\}\mapsto uc_j + ic_j$ is injective. For a fixed such $u\in \N$, making the change of variables $n\mapsto n+ um$, we obtain that the quantity in \eqref{eq 1 in seminorm control proof} becomes
\begin{align*}
     &\lim_{N\to \infty}\lim_{M\to \infty}  \norm{ \frac{1}{|\Psi_M|} \sum_{m \in \Psi_M} \frac{1}{|\Psi_N|} \sum_{n \in \Psi_N} \left( \prod_{j=1}^{k} S^{c_jn} \left( \prod_{i=1}^{\ell} S^{(uc_j+ic_j)m} (T^{r+ir'}f_{j})^{[j]} \right) \right) }_{L^2(\tau)}\\
     &\leq \lim_{N\to \infty}\frac{1}{|\Psi_N|} \sum_{n \in \Psi_N} \lim_{M\to \infty}\norm{ \frac{1}{|\Psi_M|} \sum_{m \in \Psi_M}  \left( \prod_{j=1}^{k} S^{c_jn} \left( \prod_{i=1}^{\ell} S^{(uc_j+ic_j)m} (T^{r+ir'}f_{j})^{[j]} \right) \right) }_{L^2(\tau)}\\
    &\leq C_{k,\ell}'   \lim_{N\to \infty}\frac{1}{|\Psi_N|} \sum_{n \in \Psi_N} \min \left\{ \norm{S^{c_jn}(T^{r+ir'}f_{j})^{[j]}}_{U^{k \cdot \ell}(X^k,\tau, S)} \colon 1 \leq j \leq k , 1\leq i \leq \ell \right\} \\
    &= C_{k,\ell}' \min \left\{ \norm{(T^{r+ir'}f_{j})^{[j]}}_{U^{k \cdot \ell}(X^k,\tau, S)} \colon 1 \leq j \leq k , 1\leq i \leq \ell \right\},
\end{align*}
 where the last inequality and the constant $C_{k,\ell}'>0$ come from \cite[Lemma 6.7 (iii)]{kmrr25}.
Finally, by \cite[Lemma $6.7$(iv)]{kmrr25}, there exists a constant 
$C_{k,\ell}''>0$ such that
$$\norm{(T^{r+ir'}f_{j})^{[j]}}_{U^{k \cdot \ell}(X^k,\tau, S)} = \norm{T^{r+ir'}f_{j}}_{U^{k \cdot \ell}(X,\tau_j, T^j)} \leq C_{k,\ell}'' \norm{f_{j}}_{U^{k \cdot \ell}(X,\mu, T)}$$
and so we conclude the result by defining $C_{k,\ell}=C_{k,\ell}'\cdot C_{k,\ell}''>0$.
\end{proof}

The fact that we can project to the infinite-step pronilfactor allows 
us to prove the following result. 

\begin{lemma}\label{mrp-lemma}
    If $k,\ell \in \N$ with $k\geq 2$, $F\in C(X^{k})$ and $(\Phi_N)_{N\in \N}$ is a \Folner\ sequence, then  
    \begin{align*}
      \lim_{M\to \infty} \limsup_{N\to \infty}\Bigg|\Bigg|& \frac{1}{|\Phi_M|} \sum_{m\in \Phi_N} \frac{1}{|\Phi_N|} \sum_{n\in \Phi_N}  \prod_{i=1}^{\ell} T_\Delta^{n+im}  (\1 \otimes F\otimes \1^{\ell}) - \\
      &\frac{1}{|\Phi_M|} \sum_{m\in \Phi_M} \frac{1}{|\Phi_N|} \sum_{n\in \Phi_N} \prod_{i=1}^{\ell} T_\Delta^{n+im}  (\1^{1+i} \otimes F\otimes \1^{\ell-i})\Bigg|\Bigg|_{L^2(\sigma_{k+\ell})}=0. 
    \end{align*}

\end{lemma}

\begin{proof}
We follow the proof of \cite[Theorem 6.2]{kmrr25}. By an approximation argument, we can assume $F=f_1\otimes \cdots \otimes f_k$ for $f_1,\ldots,f_k\in C(X)$. Set $G=1\otimes g_1\otimes \cdots \otimes g_k$ where $g_i=\E( f_i | Z_\infty)$ for each $i\in [k]$. Using \cref{projection-lemma-for-qmri} and approximating $g_i$ by continuous functions since the marginals of $\xi$ satisfy $\xi_i \leq i m_{Z_{\infty}}$, we see that it is enough to prove the statement in $Z_\infty$ with $F=f_1\otimes \cdots \otimes f_k$ for $f_1,\ldots,f_k\in C(Z_\infty)$ and $a \in Z_{\infty}$. Using these reductions, by unique ergodicity of orbit closures in a pronilsystem, and the definition of $\xi_{k+\ell}$ it suffices therefore to show that
$$\norm{ \frac{1}{|\Phi_M|} \sum_{m\in \Phi_M} \frac{1}{|\Phi_N|} \sum_{n\in \Phi_N} \left[ \prod_{i=1}^{\ell} T_\Delta^{n+im}  (F\otimes \1^{\ell}) -\prod_{i=1}^{\ell} T_\Delta^{n+im}  (\1^{i} \otimes F\otimes \1^{\ell-i}) \right]}^2_{L^2(\xi_{k+\ell})}$$
vanishes as $N,M\to \infty$. The last expression becomes
\begin{align*}
      &\lim_{J\to \infty}\frac{1}{|\Phi_J|}\sum_{j\in \Phi_J} \Bigg| \frac{1}{|\Phi_M|}\sum_{m\in \Phi_M} \frac{1}{|\Phi_N|} \sum_{n\in \Phi_N}  \prod_{i=1}^{\ell} T_\Delta^{n+im}  F(T^ja,\ldots,T^{kj}a) \\
      &- \frac{1}{|\Phi_M|}\sum_{m\in \Phi_M} \frac{1}{|\Phi_N|} \sum_{n\in \Phi_N}\prod_{i=1}^{\ell} T_\Delta^{n+im}  F (T^{(i+1)j}a,\ldots, T^{(i+k)j}a)\Bigg|^2 \\
      &=\lim_{J\to \infty}\frac{1}{|\Phi_J|}\sum_{j\in \Phi_J} \Bigg| \frac{1}{|\Phi_M|}\sum_{m\in \Phi_M} \frac{1}{|\Phi_N|} \sum_{n\in \Phi_N}  \prod_{i=1}^{\ell} T_\Delta^{n}  F(T^{j+im}a,\ldots,T^{kj+im}a) \\
      &- \frac{1}{|\Phi_M|}\sum_{m\in \Phi_M} \frac{1}{|\Phi_N|} \sum_{n\in \Phi_N}\prod_{i=1}^{\ell} T_\Delta^{n}  F (T^{(i+1)j+im}a,\ldots, T^{(i+k)j+im}a)\Bigg|^2 
\end{align*}
We are thus left with showing that 
\begin{align*}
\lim_{N,M\to \infty}&\lim_{J\to \infty}\frac{1}{|\Phi_J|}\sum_{j\in \Phi_J} \Bigg| \frac{1}{|\Phi_M|}\sum_{m\in \Phi_M} \frac{1}{|\Phi_N|} \sum_{n\in \Phi_N}  \prod_{i=1}^{\ell} T_\Delta^{n}  F(T^{j+im}a,\ldots,T^{kj+im}a) \\
      &- \frac{1}{|\Phi_M|}\sum_{m\in \Phi_M} \frac{1}{|\Phi_N|} \sum_{n\in \Phi_N}\prod_{i=1}^{\ell} T_\Delta^{n}  F (T^{(i+1)j+im}a,\ldots, T^{(i+k)j+im}a)\Bigg|^2=0.  
\end{align*}
Arguing as in \cref{appendix lemam to interchange limits}, the unique ergodicity allows to exchange the order of the limits, so that the above becomes
\begin{align*}
\lim_{N,J\to \infty}&\lim_{M\to \infty}\frac{1}{|\Phi_J|}\sum_{j\in \Phi_J} \Bigg| \frac{1}{|\Phi_M|}\sum_{m\in \Phi_M} \frac{1}{|\Phi_N|} \sum_{n\in \Phi_N}  \prod_{i=1}^{\ell} T_\Delta^{n}  F(T^{j+im}a,\ldots,T^{kj+im}a) \\
      &- \frac{1}{|\Phi_M|}\sum_{m\in \Phi_M} \frac{1}{|\Phi_N|} \sum_{n\in \Phi_N}\prod_{i=1}^{\ell} T_\Delta^{n}  F (T^{(i+1)j+im}a,\ldots, T^{(i+k)j+im}a)\Bigg|^2=0.  
\end{align*}
Making the change of variables $m\mapsto m-j$ allows us to conclude, because 
$$ \Bigg| \frac{1}{|\Phi_M|} \sum_{m\in \Phi_M}  \prod_{i=1}^{\ell} T_\Delta^{n}  F(T^{j+im}a,\ldots,T^{kj+im}a) - \frac{1}{|\Phi_M|} \sum_{m\in \Phi_M+j}\prod_{i=1}^{\ell} T_\Delta^{n}  F (T^{j+im}a,\ldots, T^{kj+im}a)\Bigg|^2,$$
vanishes as $M\to \infty$, for any $n,j\in \N$, by the asymptotic invariance of \Folner\ sequences.
\end{proof}

We are finally in a position to prove \cref{multiple-recurrence-T-Delta-uncoupled}.
\begin{proof}[Proof of \cref{multiple-recurrence-T-Delta-uncoupled}]
We want to show that 
$$ \liminf_{M\to\infty}\liminf_{N\to\infty}\frac{1}{|\Phi_M|} \sum_{m\in \Phi_M}  \frac{1}{|\Phi_N|} \sum_{n\in \Phi_N} \int (F\otimes \1^{\otimes \ell}) \cdot \prod_{i=1}^{\ell}  (T_{\Delta}^{n+im} (\1^{\otimes i} \otimes F\otimes \1^{\otimes \ell -i})) \diff\sigma_{k+\ell}>0.$$ 
Under the assumptions of \cref{multiple-recurrence-T-Delta-uncoupled}, 
we can apply \cref{mrp-lemma} and the above averages become
$$ \liminf_{M\to\infty}\liminf_{N\to\infty}\frac{1}{|\Phi_M|} \sum_{m\in \Phi_M}  \frac{1}{|\Phi_N|} \sum_{n\in \Phi_N} \int (F\otimes \1^{\otimes \ell}) \cdot \prod_{i=1}^{\ell}  (T_{\Delta}^{n+im} (F\otimes \1^{\otimes \ell})) \diff\sigma_{k+\ell}.$$
Note that we can rewrite these averages as 
$$ \liminf_{M\to\infty}\liminf_{N\to\infty}\frac{1}{|\Phi_M|} \sum_{m\in \Phi_M}  \frac{1}{|\Phi_N|} \sum_{n\in \Phi_N} \int F \cdot \prod_{i=1}^{\ell}  T_{\Delta}^{n+im} F \diff \sigma_k$$
and the latter expression is positive by \cref{T-Delta-Furstenberg-cor}.
\end{proof}

\subsection{Proof of \cref{topological pronilfactors 1}}

Before presenting its proof, we restate \cref{topological pronilfactors 1} for 
the convenience of the reader.

\begin{named}{\cref{topological pronilfactors 1}}{}
Let $(X,\mu,T)$ be an ergodic system with topological infinite-step 
pronilfactor. Let $\Phi=(\Phi_N)_{N\in \N}$ be a \Folner\ sequence, 
let 
$a\in \gen(\mu,\Phi)$ and let $\sigma$ be the measure defined in \cref{definition of sigma}. Then, 
for every open set $E\subset X$ with $\mu(E)>0$ there exist sequences $(b_n)_{n\in \N}, (s_k)_{k\in \N}, (t_k)_{k\in \N}\subset \N$ such 
that 
    \begin{equation}\label{eq-measure-conj-1}
\sigma_{k+1}\left(\bigcap_{i=1}^k  \prod_{j=0}^{k+1}  \bigcap_{\ell=\max(j,1)}^{i+1}  T^{-\left( t_i+\ell s_i+ B_k^{\oplus \ell-j} \right)}  E   \right)>0
    \end{equation}
for every $k\in \N$, where $B_k=\{b_1,\ldots,b_{k}\}$.  
\end{named}

The proof of \cref{topological pronilfactors 1} can be summarized in the following diagram.

\begin{figure}[h!]
\centering
\resizebox{1\textwidth}{!}{%
\begin{tikzpicture}[
    font=\LARGE,
    >=Stealth,
    node distance=2cm and 3cm,
    box/.style={rectangle, draw, minimum width=4cm, minimum height=1.5cm, align=center}
]
\node[box] (A) {Initialization step};
\node[box, below=of A] (B) {Left-progressiveness};
\node[box, below=of B] (C) {Multiple \\ Right-progressiveness};

\draw[->] (A) -- (B);
\draw[->] (B.east) -- ++(2cm,0) |- (C.east);
\draw[->] (C.west) -- ++(-2cm,0) |- (B.west);

\draw[->, dashed] (A.east) -- ++(4cm,0) node[right]{$\begin{matrix}
    \text{Finds } t_1  \text{ such that} \\
    \sigma_2(X \times T^{-t_1}E \times T^{-t_1 }E) >0
\end{matrix} $};
\draw[->, dashed] (B.east) ++(2cm,-2cm) -- ++(2cm,0) node[right]{$\begin{matrix}
    \text{Finds } b_k, t_k, \text{ and } s_k  \text{ such that} \\
   \sigma_{k+1}\left(\bigcap_{i=1}^k  \prod_{j=0}^{k+1}  \bigcap_{\ell=\max(j,1)}^{i+1}  T^{-\left( t_i+\ell s_i+ B_k^{\oplus \ell-j} \right)}  E   \right)>0
\end{matrix} $};
\end{tikzpicture}}%
\caption{Diagram of the proof of \cref{topological pronilfactors 1}}
\label{fig:1}
\end{figure}
\begin{proof}
 Recall that $\sigma$ is progressive (see \cref{progressive proof}) as in Definition \ref{definition progressive measures in XN}. Take $s_1=0$. As $\mu(E)>0$, by \cref{Initialization-step}, there is $t_1\in \N$ such that 
    $$\sigma_2(X \times T^{-t_1}E \times T^{-t_1 }E) >0. $$
    By left-progressiveness of $\sigma_2$, there is $b_1\in \N$ such that 
    \begin{equation*}
     \sigma_2(T^{-(t_1+b_1)}E \times (T^{-(t_1
    +b_1)}E \cap T^{-t_1 }E) \times T^{-t_1 }E )>0. 
    \end{equation*}  
    If we set $B_1=\{b_1\}$, the previous is equivalent to 
    $$\sigma_{2}\left(\bigcap_{i=1}^1  \prod_{j=0}^{2}  \bigcap_{\ell=\max(j,1)}^{i+1}  T^{-\left( t_i+\ell s_i+ B_1^{\oplus \ell-j} \right)}  E   \right)>0, $$
    in accordance with the notation in \eqref{eq-measure-conj-1}.
    If we let $\tilde{E}_1=X \times T^{-(t_1+b_1)}E \times T^{-t_1 }E$ and $E_1=T^{-(t_1+b_1)}E \times (T^{-(t_1+b_1)}E \cap T^{-t_1 }E) \times T^{-t_1 }E$, then we can use the multiple right-progressiveness of $\sigma$ to find $s_2>s_1$ and $t'\in \N$ such that 
    $$\sigma_{5}((E_1\times X^3) \cap T_\Delta^{-(t'+s_2)} (X\times \tilde{E}_1\times X^2)\cap T_\Delta^{-(t'+2s_2)} (X^2\times \tilde{E}_1\times X)\cap T_\Delta^{-(t'+3s_2)} (X^3\times \tilde{E}_1 ) )>0. $$
    If we let $\tilde{t}_2=t'+t_1$, 
    observe that the set 
    $$ T_\Delta^{-(t'+s_2)} (X\times \tilde{E}_1\times X^2)\cap T_\Delta^{-(t'+2s_2)} (X^2\times \tilde{E}_1\times X)\cap T_\Delta^{-(t'+3s_2)} (X^3\times \tilde{E}_1 )$$
    is actually equal to the set
    \begin{align*}
      &(X^2 \times T^{-(\tilde{t}_2+s_2+b_1)}E \times T^{-(\tilde{t}_2+s_2)} E \times X^2)\cap (X^3 \times T^{-(\tilde{t}_2+2s_2+b_1)}E \times  T^{-(\tilde{t}_2+2s_2)} E\times X)\\& \cap (X^4 \times T^{-(\tilde{t}_2+3s_2+b_1)}E \times T^{-(\tilde{t}_2+3s_2)}E).  
    \end{align*}
    Thus, we have that the $\sigma_5$ measure of the intersection of $(E_1\times X^3)$ with the set
    $$  X^2 \times T^{-(\tilde{t}_2+s_2+b_1)}E \times (T^{-(\tilde{t}_2+2s_2+b_1)} E\cap T^{-(\tilde{t}_2+s_2)}E) \times  (T^{-(\tilde{t}_2+3s_2+b_1)} E\cap T^{-(\tilde{t}_2+2s_2)}E) \times T^{-(\tilde{t}_2+3s_2)}E$$
    is positive. By left-progressiveness (applied twice, only to the latter product set), we can find $t_2'$ so that if $t_2=t_2'+\tilde{t}_2$, then the $\sigma_3$ measure of the intersection of $(E_1\times X)$ with
    $$ T^{-(t_2+s_2+b_1)}E \times (T^{-(t_2+2s_2+b_1)} E\cap T^{-(t_2+s_2)}E) \times  (T^{-(t_2+3s_2+b_1)} E\cap T^{-(t_2+2s_2)}E) \times T^{-(t_2+3s_2)}E$$
    is positive. Applying left progressiveness once again to the last intersection and using the definition of $E_1$, we find $b_2>b_1$ such that the $\sigma_3$ measure of     
    \begin{align*}
    &  \left( (T^{-(t_1
    +b_1)}E \cap T^{-(t_1
    +b_2)}E \cap T^{-(t_1
    +b_2+b_1)}E) \times (T^{-(t_1
    +b_2)}E \cap  T^{-(t_1+b_1)}E \cap T^{-t_1 }E) \times T^{-t_1 }E \times X \right)  \cap \\
   &\Big( (T^{-(t_2+s_2+b_1)}E \cap T^{-(t_2+2s_2+b_1+b_2)} E\cap  T^{-(t_2+s_2+b_2)}E)  \times  \\ &   (T^{-(t_2+2s_2+b_1)} E\cap T^{-(t_2+s_2)}E \cap T^{-(t_2+3s_2+b_1+b_2)} E\cap T^{-(t_2+2s_2+b_2)} E )\times \\
   & (T^{-(t_2+3s_2+b_1)} E\cap T^{-(t_2+2s_2)}E \cap T^{-(t_2+3s_2+b_2)}E) \times T^{-(t_2+3s_2)}E \Big) 
    \end{align*}
is positive. If we let $B_2=\{b_1,b_2\}$, the previous is equivalent to the following:
     $$\sigma_{3}\left(\bigcap_{i=1}^2  \prod_{j=0}^{3}  \bigcap_{\ell=\max(j,1)}^{i+1}  T^{-\left( t_i+\ell s_i+ B_2^{\oplus \ell-j} \right)}  E   \right)>0. $$
In this fashion we   construct the sequences of the statement inductively. Assume that we have created $b_1<\cdots<b_k$, $t_1<\cdots<t_k$ and 
    $ s_1<\cdots<s_k$ such that 
    \begin{equation}\label{eq-measure-conj-2}
\sigma_{k+1}\left(\bigcap_{i=1}^k  \prod_{j=0}^{k+1}  \bigcap_{\ell=\max(j,1)}^{i+1}  T^{-\left( t_i+\ell s_i+ B_k^{\oplus \ell-j} \right)}  E   \right)>0,
    \end{equation}
    where $B_k=\{b_1,\ldots,b_k\}$.
Call $$E_k=   \prod_{j=0}^{k+1}    T^{-\left( t_k+(k+1) s_k+ B_k^{\oplus k+1-j} \right)}  E =\prod_{j=0}^{k+1}    T^{-\left( \tilde{t}_k+ B_k^{\oplus k+1-j} \right)}  E ,  $$ 
where $\tilde{t}_k=t_k+(k+1) s_k $. Notice that
$$ E_k\supset   \prod_{j=0}^{k+1}  \bigcap_{\ell=\max(j,1)}^{k+1}  T^{-\left( t_k+\ell s_k+ B_k^{\oplus \ell-j} \right)}  E  \supset  \bigcap_{i=1}^k  \prod_{j=0}^{k+1}  \bigcap_{\ell=\max(j,1)}^{i+1}  T^{-\left( t_i+\ell s_i+ B_k^{\oplus \ell-j} \right)}  E .$$
Using multiple right-progressiveness in \eqref{eq-measure-conj-2} with $E_k$ 
solely, we find $s_{k+1}>s_k$ and $t'\in \N$ such that the $\sigma_{2k+3}$ measure of the intersection of
$$\left(\bigcap_{i=1}^k  \prod_{j=0}^{k+1}  \bigcap_{\ell=\max(j,1)}^{i+1}  T^{-\left( t_i+\ell s_i+ B_k^{\oplus \ell-j} \right)}  E \right) \times X^{k+2}$$
with 
$$T_\Delta^{-(t'+s_{k+1})} (X\times  E_k\times X^{k+1}) \cap \cdots \cap T_{\Delta}^{-(t'+(k+2)s_{k+1})} (X^{k+2} \times E_k)$$
is positive. We notice that the set 
$$X^{k+1}\times \prod_{j=0}^{k+2}  \bigcap_{\ell=\max(j,1)}^{k+2}  T^{-\left( \tilde{t}_k+t'+\ell s_{k+1}+ B_k^{\oplus \ell-j} \right)}  E  $$
contains the set 
\begin{equation}\label{set-MRP}
     T_\Delta^{-(t'+s_{k+1})} (X\times  E_k\times X^{k+1}) \cap \cdots \cap T_{\Delta}^{-(t'+(k+2)s_{k+1})} (X^{k+2} \times E_k).
\end{equation}
as the set in the $(k+2+j)$-th coordinate of the set in \eqref{set-MRP} is
$$\bigcap_{\ell=\max(j,1)}^{k+2}  T^{-\left( \tilde{t}_k+t'+\ell s_{k+1}+ B_k^{\oplus \ell-j} \right)}  E,  $$
for each $j\in \{0,\ldots,k+2\}$. To see this, notice, for example, that the $(k+2)$-th coordinate of the set in \eqref{set-MRP} is given by intersecting the $(k+2-\ell)$-th coordinates of $E_k$ under $T^{(-t'+\ell s_{k+1})}$ for $\ell=1,\ldots,k+1$, and so it equals 
$$\bigcap_{\ell=1}^{k+1} T^{-(t'+\ell s_{k+1})}  T^{-\left( \tilde{t_k} + B_k^{\oplus k+1-(k+2-\ell-1)} \right)}  E = \bigcap_{\ell=1}^{k+2} T^{-\left( \tilde{t_k}+t'+\ell s_{k+1} + B_k^{\oplus \ell} \right)}  E,$$
since $B_k^{k+2}=\emptyset$.

Hence, by left-progressiveness applied $k+1$ times and disregarding the last coordinates, we can find $t_{k+1}>t_k$ (which groups each $\tilde{t}_k,t'$ and the shift given by the $(k+1)$ applications of left-progressiveness) such that 
\begin{align*}
  &\sigma_{k+2} \left(\left(\bigcap_{i=1}^k  \prod_{j=0}^{k+2}  \bigcap_{\ell=\max(j,1)}^{i+1}  T^{-\left( t_i+\ell s_i+ B_k^{\oplus \ell-j} \right)}  E \right) \cap \left(\prod_{j=0}^{k+2}  \bigcap_{\ell=\max(j,1)}^{k+2}  T^{-\left( t_{k+1}+\ell s_{k+1}+ B_k^{\oplus \ell-j} \right)}  E  \right)\right)  \\
  &=\sigma_{k+2 }\left( \bigcap_{i=1}^{k+1}  \prod_{j=0}^{k+2}  \bigcap_{\ell=\max(j,1)}^{i+1}  T^{-\left( t_i+\ell s_i+ B_k^{\oplus \ell-j} \right)}  E  \right)>0.
    \end{align*}
    Now, by left-progressiveness once again, we obtain $b_{k+1}>b_k$ such that the $\sigma_{k+2}$ measure of
    \begin{align*}
      &\left(\bigcap_{i=1}^{k+1}  \prod_{j=0}^{k+2}  \bigcap_{\ell=\max(j,1)}^{i+1}  T^{-\left( t_i+\ell s_i+ B_k^{\oplus \ell-j} \right)}  E \right) \cap \left(\bigcap_{i=1}^{k+1}  \prod_{j=0}^{k+2}  \bigcap_{\ell=\max(j+1,1)}^{i+1}  T^{-\left( t_i+\ell s_i+b_{k+1}+ B_k^{\oplus \ell-j-1} \right)}  E   \right)\\
      &=\left( \bigcap_{i=1}^{k+1}  \prod_{j=0}^{k+2} \left( \bigcap_{\ell=\max(j,1)}^{i+1}  T^{-\left( t_i+\ell s_i+ B_k^{\oplus \ell-j} \right)}  E \cap  \bigcap_{\ell=j+1}^{i+1}  T^{-\left( t_i+\ell s_i+b_{k+1}+ B_k^{\oplus \ell-j-1} \right)}  E  \right)  \right) 
    \end{align*} 
    is positive.
    If we set $B_{k+1}=B_k\cup \{b_{k+1}\}$, we then have that
    $$
        0<\sigma_{k+2}\left( \bigcap_{i=1}^{k+1}  \prod_{j=0}^{k+2}  \bigcap_{\ell=\max(j,1)}^{i+1}  T^{-\left( t_i+\ell s_i+ B_{k+1}^{\oplus \ell-j} \right)}  E  \right)
    $$
    concluding the induction. We have thus constructed sequences $(b_n)_{n\in \N}$ and $(t_k)_{k\in \N}, (s_k)_{k\in \N}$ such that \eqref{eq-measure-conj-1} holds for each $k\in \N$ and the proof is complete.
\end{proof}

 It is possible to find a set $B\subset \N$ that satisfies the statements of both \cref{theo-A} and \cref{theo-B} at the same time. As \cref{fig:2} and \cref{fig:1} show, the proofs of these theorems only differ in the iterative application of multiple right-progressiveness to find the shifts $t_k,s_k$ for \cref{theo-B}, and right-progressiveness to find the shift $t_k$ for \cref{theo-A}. Thus, we can use both multiple-right progressiveness and right progressiveness, in addition to left-progressiveness, to have, in each iteration, three shifts $t_k, t_k'$ and $s_k$ such that
    \begin{align*}
    \sigma_{k+\ell+1}\Bigg[& \Bigg( \bigcap_{i=1}^{k+1}  \prod_{j=0}^{k+\ell+1}\bigcap_{h=\max(j,i)}^{i+\ell}  T^{-\left( t_i'+ B_k^{\oplus h-j} \right)}  E  \Bigg)    \cap\\
    &\Bigg(\bigcap_{i=1}^k  \Bigg(\prod_{j=0}^{k+1}  \bigcap_{\ell=\max(j,1)}^{i+1}  T^{-\left( t_i+\ell s_i+ B_k^{\oplus \ell-j} \right)}  E   \Bigg) \times X^\ell \Bigg) \Bigg] >0, 
    \end{align*}
    which would imply that  the same infinite set $B\subset \N$ satisfies the conclusions of \cref{theo-B} and \cref{theo-A} simultaneously. Thus, we obtain the following result. 
\begin{prop}\label{Mix-of-Theo-A-and-B}
Let $A\subset\N$ be a set with $\diff^*(A)>0$ and $\ell\geq 0$ an integer. Then, there exist an infinite set $B\subset \N$ and sequences $(\tilde{t}_{k})_{k\in \N}$, $ (t_k)_{k\in \N}$, $(s_k)_{k\in \N} \subset \N$ such that 
$$ \bigg\{ \sum_{n\in F} n \colon \ F\subset B \text{ with } k\leq |F|\leq k+ \ell \bigg\} \subset A-\tilde{t}_{k}, \ \text{for every}\ k\in \N,$$
and
$$ \bigg\{ \sum_{n\in F} n \colon \ F\subset B+s_k \text{ with } 1\leq |F|\leq k \bigg\} \subset A-t_k ,\ \text{for every}\ k\in \N,$$
simultaneously.
\end{prop}

\appendix





\section{Progressive measures, averages and seminorms} \label{appendix progresive measures in Zk}

\subsection{Progressive measures coming from finite-step nilsystems are distinct}

For an ergodic system $(X,\mu,T)$ with continuous pronilfactors, the progressive measures $\sigma^{Z_{k}}$ for $k \geq 1$ are defined in \cite{kmrr25} as follows: Consider $a \in \gen (\mu, \Phi)$ and $Z_{k}$ the $k$-step pronilfactor with continuous factor map $\pi_{k} \colon X \to Z_k $. Define $\xi^{Z_k} \in \cM(Z_k^{k})$ to be the Haar measure of the orbit closure of $(\pi_k(a), \ldots, \pi_k(a))$ under the transformation $T \times \cdots \times T^{k}$ and for $f_0, f_1, \ldots, f_{k} \in C(X)$ the measure $\sigma^{Z_k} \in \cM(X^{k+1})$ is defined via
\begin{equation} \label{eq progressive measure in Zk}
    \int_{X^{k+1}} f_0 \otimes f_1 \otimes \cdots \otimes f_{k} \diff \sigma^{Z_k} = f_0(a) \int_{Z_k^{k}}  \E(f_1 \mid Z_k) \otimes \cdots \otimes \E(f_{k} \mid Z_k) \diff \xi^{Z_k}.
\end{equation}

Notice that even though it is not explicit in the notation, the measure $\sigma^{Z_{k}}$ depends on the choice of the point $a \in \gen (\mu, \Phi)$.

In this appendix we want to illustrate that there are cases where if $P_{k+1} \colon X^{k+2} \to X^{k+1}$ is the projection to the first $k+1$ coordinates, then $P_{k+1} \sigma^{Z_{k+1}}$ and $\sigma^{Z_{k}}$ are distinct (see also \cite[Example 3.19]{ackelsberg_jamnesham2025equidistribution}).

Fix $s \geq 2$ and let $X = \T^s $ and $T \colon X \to X$ be the affine transformation given by $T(x_1, \ldots,x_s) = (x_1 + \alpha, x_2 + 2 x_1 + \alpha, \ldots, x_s + s x_{s-1} + \binom{s}{2} x_{s-2} + \cdots + s x_1 + \alpha)$. This transformation has been repeatedly used in the literature; see, for example, \cite[Proposition 3.11]{Furstenberg81}. In particular if $a= (0,\ldots, 0)$ then $T^n a = (n \alpha, n^2 \alpha, \ldots, n^s \alpha)$

Notice that the factors $Z_k(X)$ for $k \leq s$ in this case are the systems $Z_k = \T^k$ with factor maps $\pi_k \colon X \to Z_k$ given by projecting onto the first $k$ coordinates. Moreover, for every $k \leq s$, $\sigma^{Z_{k+1}}$ is the Haar measure on the manifold
\begin{equation*}
    \Omega_k = \left\{ \left(\begin{array}{c}
         v_0 \\
         v_1 \\
         \vdots \\
         v_{k+1}
    \end{array} \right)
     \in (\T^s)^{k+2} \middle| \begin{array}{c}
          v_0 = a, \\
          v_r = ( r t_1, r^2 t_2, \ldots, (r+1)^rt_k, u_{r,1}, \ldots,u_{r,s+1-k}),  \\
           t_1, \dots, t_k \in \T \quad \text{and} \\
           \quad u_{r,1}, \ldots, u_{r,s+1-k} \in \T \quad \text{for } r=1,\dots k+1
     \end{array} \right\}
\end{equation*}
where we recall that $a= (0,\ldots,0) \in \T^s=X$. 
For instance, when $s=3$,

\begin{align*}
    \Omega_1 &= \{(0,0,0, \ t, u,v, \ 2t, w,w'  ) \mid t, u,v,w,w' \in \T \} \subset X^{1+2} \\ \\
   \Omega_2  &= \{(0,0,0, \ t,r, u, \ 2t, 4r, v, \ 3t,9r,w ) \mid t,r, u,v,w \in \T \} \subset X^{2+2} \\ \\
    \Omega_3  &= \{(0,0,0,\ t,r, s, \ 2t, 4r, 8s,\ 3t,9r, 27s,\ 4t, 16r,64s) \mid t,r, s \in \T \} \subset X^{3+2} \\
\end{align*}

 Notice that $P_{k} ( \Omega_k)$, the image of $\Omega_k$ under the projection $P_{k} \colon X^{k+2} \to X^{k+1}$ is strictly contained in $\Omega_{k-1}$, so in particular they are distinct. Since Haar measures of nilmanifold are fully supported, one concludes that the measures $P_{k+1} \sigma^{Z_k}$ and $\sigma^{Z_{k-1}}$ are distinct for all $k \leq s$.

 For the same reason, in this example, the measures $\sigma^{Z_k}$ are not ergodic under the action of $Id \times T \times \cdots \times T^k$ for $k <s$. In fact, the nilmanifold $P_{s,k}(\Omega_s)$, that is the image of $\Omega_s$ under the projection map $P_{s,k} \colon X ^{s+2} \to X^{k+2} $ to the first $k+2$ coordinates, is a minimal component of $\Omega_k$ for the transformation $Id \times T \times \cdots \times T^k$. This shows that $\Omega_k$ itself is not minimal for $Id \times T \times \cdots \times T^k$ and hence its Haar measure $\sigma^{Z_k}$ is not ergodic.


\subsection{Double average for seminorm control} \label{subsec ex double average}

Here we use the same family of examples to prove that the seminorms control is only guaranteed in the double averaging version of \cref{projection-lemma-for-qmri}. More precisely, we build an example where if one consider a single average, the $L^2$ norm is no longer controlled by the seminorms of $f_1,f_2$.  

Consider $X= \T^3$ with the transformation $T\colon \T^3 \to \T^3$ given by $(x,y,z) \mapsto (x + \alpha, y + 2x + \alpha, z+3y+3x+\alpha)$ for $\alpha$ irrational. Let also $\boldsymbol{0} = (0,0,0)$ and
 \begin{align*}
    \Omega &= \overline{Orb_{T\times T^2} (\boldsymbol 0, \boldsymbol0)} = \{(t, s, r, 2t, 4s,9r  ) \mid t, s,r \in \T \} \subset X^2
\end{align*}
with its Haar measure $\xi$. Then, $(\Omega, \xi,T\times T^2)$ is a minimal and ergodic system. Notice that, for $F \in C(\Omega)$ and $(t_0, s_0, r_0, 2t_0, 4s_0 ,9r_0 ) \in \Omega$, the conditional expectation onto $W=Z_2(\Omega, \xi, T\times T^2) $ is given by
\begin{equation*}
    \E( F \mid W) (t_0, s_0, r_0, 2t_0, 4s_0 ,9r_0 ) = \int_{\T} F(t_0, s_0, r, 2t_0, 4s_0 ,9r ) \diff m(r).
\end{equation*}
In particular if $f_1, f_2 \in C(\T^3)$ are given by $f_1=\1_{\T^2 \times (0,1/9^2)} + 1$ and $f_2 = \1_{\T^2 \times (1/9,2/9)} +1$ we have that
\begin{align*}
   \E(f_1\mid Z_2) =& \E( \1_{\T^2 \times (0,1/9^2)} \mid Z_2) + 1 = 1 + 1/9^2 \\
   \E(f_2\mid Z_2) =&  \E( \1_{\T^2 \times (1/9,2/9)} \mid Z_2) + 1 = 1+ 1/9  \\
   \E(f_1 \otimes f_2\mid W) =& \E( \1_{\T^2 \times (0,1/9^2)} \otimes \1_{\T^2 \times (1/9,2/9)} \mid W) +  \E( \1_{\T^2 \times (0,1/9^2)} + \1_{\T^2 \times (1/9,2/9)} \mid Z_2)  + 1 \\
   =& 0 + 1/9^2 + 1/9 +1
\end{align*}
In particular $\E(f_1 \otimes f_2 \mid W) - \E(f_1 \mid Z_2) \otimes \E(f_2 \mid Z_2) \neq 0 $. Moreover, we have that
\begin{equation*}
    \lim_{N \to \infty} \norm{\frac{1}{N} \sum_{n \leq N}(T^n f_1 \cdot T^{2n} f_1) \otimes (T^n f_2 \cdot T^{2n} f_2)  }_{L^2(\xi)}
\end{equation*}
is controlled by 
\begin{equation*}
    \lim_{N \to \infty} \norm{\frac{1}{N} \sum_{n \leq N} \tilde T^{4n} ( f_1\otimes \1) \cdot \tilde T^{2n} ( f_1\otimes f_2) \cdot \tilde T^{n} ( \1 \otimes f_2)  }_{L^2(\xi)},
\end{equation*}
by similar arguments as in the proof of \cref{projection-lemma-for-qmri}, where here $\tilde T = T \times T^2$. By \cite{Host_Kra_nonconventional_averages_nilmanifolds:2005} we can replace it by
\begin{align*}
    \lim_{N \to \infty}& \norm{\frac{1}{N} \sum_{n \leq N} \tilde T^{4n} \E( f_1\otimes \1 \mid W) \cdot \tilde T^{2n} \E( f_1\otimes f_2 \mid W) \cdot \tilde T^{n} \E( \1 \otimes f_2 \mid W)  }_{L^2(\xi)} \\=& (1/9^3  +1)(1/9^3 + 1/9 +1) (1/9 +1)
\end{align*}
On the other hand, 
\begin{align*}
    \lim_{N \to \infty}& \norm{\frac{1}{N} \sum_{n \leq N} \tilde T^{4n} (\E( f_1\mid Z_2)\otimes \1) \cdot \tilde T^{2n} (\E( f_1 \mid Z_2) \otimes \E(f_2 \mid Z_2)) \cdot \tilde T^{n} ( \1 \otimes \E(f_2 \mid Z_2))  }_{L^2(\xi)} \\=& (1/9^3  +1)^2 (1/9 +1)^2.
\end{align*}

\small{
\bibliographystyle{abbrv}
\bibliography{refs}

\begin{thebibliography}{10}

\bibitem{ackelsberg2025polynomial_patterns_rationals}
E.~Ackelsberg.
\newblock Infinite polynomial patterns in large subsets of the rational numbers.
\newblock {\em arXiv preprint arXiv:2506.19667}, 2025.

\bibitem{ackelsberg_jamnesham2025equidistribution}
E.~Ackelsberg and A.~Jamneshan.
\newblock Equidistribution in 2-nilpotent polish groups and triple restricted sumsets.
\newblock {\em arXiv preprint arXiv:2504.07865}, 2025.

\bibitem{charamaras_kousek_mountakis_radic2025BBingroups}
D.~Charamaras, I.~Kousek, A.~Mountakis, and T.~Radi{\'c}.
\newblock Infinite unrestricted sumsets in subsets of abelian groups with large density.
\newblock {\em arXiv preprint arXiv:2504.08649}, 2025.

\bibitem{Charamaras_Mountakis_2025}
D.~Charamaras and A.~Mountakis.
\newblock Finding product sets in some classes of amenable groups.
\newblock {\em Forum of Mathematics, Sigma}, 13, 2025.

\bibitem{diNasso_Golbring_Jin_Leth_Lupini_Mahlburg2015sumset}
M.~Di~Nasso, I.~Goldbring, R.~Jin, S.~Leth, M.~Lupini, and K.~Mahlburg.
\newblock On a sumset conjecture of erd{\H{o}}s.
\newblock {\em Canadian Journal of Mathematics}, 67(4):795--809, 2015.

\bibitem{Erdos77}
P.~Erd{\H{o}}s.
\newblock Problems and results on combinatorial number theory. {III}.
\newblock pages 43--72. Lecture Notes in Math., Vol. 626, 1977.

\bibitem{erdHos1980survey}
P.~Erd{\H{o}}s.
\newblock A survey of problems in combinatorial number theory.
\newblock {\em Annals of Discrete Mathematics}, 6:89--115, 1980.

\bibitem{erdos2006problems}
P.~Erd{\"o}s.
\newblock Problems and results on combinatorial number theory iii.
\newblock In {\em Number Theory Day: Proceedings of the Conference Held at Rockefeller University, New York 1976}, pages 43--72. Springer, 2006.

\bibitem{Furstenberg77}
H.~Furstenberg.
\newblock Ergodic behavior of diagonal measures and a theorem of {S}zemer\'edi on arithmetic progressions.
\newblock {\em J. d'Analyse Math.}, 31:204--256, 1977.

\bibitem{Furstenberg81}
H.~Furstenberg.
\newblock {\em Recurrence in ergodic theory and combinatorial number theory}.
\newblock Princeton University Press, Princeton, N.J., 1981.

\bibitem{Furstenberg_Katznelson78}
H.~Furstenberg and Y.~Katznelson.
\newblock An ergodic {S}zemer\'edi theorem for commuting transformations.
\newblock {\em J. Analyse Math.}, 34:275--291 (1979), 1978.

\bibitem{hernandez2025infinite}
F.~Hern{\'a}ndez.
\newblock Infinite linear patterns in sets of positive density.
\newblock {\em arXiv preprint arXiv:2505.15458}, 2025.

\bibitem{Hindman_ultrafilters}
N.~Hindman.
\newblock Ultrafilters and combinatorial number theory.
\newblock {\em Number theory, Carbondale 1979}, (Proc. Southern Illinois Conf., Southern Illinois Univ., Carbondale, Ill.,1979) Lecture Notes in Math., vol. 751, Springer, Berlin, 1979, pp. 119–184.

\bibitem{Hindman74}
N.~Hindman.
\newblock Finite sums from sequences within cells of a partition of {$N$}.
\newblock {\em J. Combinatorial Theory Ser. A}, 17:1--11, 1974.

\bibitem{host2019short}
B.~Host.
\newblock A short proof of a conjecture of erd$\backslash$" os proved by moreira, richter and robertson.
\newblock {\em Discrete Analysis}, page~19, 2019.

\bibitem{Host_Kra_nonconventional_averages_nilmanifolds:2005}
B.~Host and B.~Kra.
\newblock Nonconventional ergodic averages and nilmanifolds.
\newblock {\em Ann. of Math. (2)}, 161(1):397--488, 2005.

\bibitem{Host_Kra_nilpotent_structures_ergodic_theory:2018}
B.~Host and B.~Kra.
\newblock {\em Nilpotent structures in ergodic theory}, volume 236 of {\em Mathematical Surveys and Monographs}.
\newblock American Mathematical Society, Providence, RI, 2018.

\bibitem{kousek2025asymmetric}
I.~Kousek.
\newblock Asymmetric infinite sumsets in large sets of integers.
\newblock {\em arXiv preprint arXiv:2502.03112}, 2025.

\bibitem{kousek_radic2025BB}
I.~Kousek and T.~Radi{\'c}.
\newblock Infinite unrestricted sumsets of the form b+ b b+b in sets with large density.
\newblock {\em Bulletin of the London Mathematical Society}, 57(1):48--68, 2025.

\bibitem{Kra_Moreira_Richter_Robertson:2023}
B.~Kra, J.~Moreira, F.~Richter, and D.~Robertson.
\newblock A proof of erd{\H o}s{\textquoteright}s b + b + t conjecture.
\newblock {\em Communications of the American Mathematical Society}, 4:480--494, 2024.

\bibitem{Kra_Moreira_Richter_Robertson_problems}
B.~Kra, J.~Moreira, F.~Richter, and D.~Robertson.
\newblock Problems on infinite sumset configurations in the integers and beyond.
\newblock {\em Bulletin of the American Mathematical Society}, 62(4):537--574, 2025.

\bibitem{Kra_Moreira_Richter_Robertson:2022}
B.~Kra, J.~Moreira, F.~K. Richter, and D.~Robertson.
\newblock Infinite sumsets in sets with positive density.
\newblock {\em Journal of the American Mathematical Society}, 36(4), 2023.

\bibitem{kmrr25}
B.~Kra, J.~Moreira, F.~K. Richter, and D.~Robertson.
\newblock The density finite sums theorem.
\newblock {\em Inventiones mathematicae}, pages 1--31, 2025.

\bibitem{Moreira_Richter_Robertson19}
J.~Moreira, F.~Richter, and D.~Robertson.
\newblock A proof of a sumset conjecture of {E}rd{\H o}s.
\newblock {\em Ann. of Math. (2)}, 189(2):605--652, 2019.

\bibitem{Zorin-Kranich16}
P.~Zorin-Kranich.
\newblock Norm convergence of multiple ergodic averages on amenable groups.
\newblock http://arxiv.org/abs/1111.7292v8.

\end{thebibliography}

}
\bigskip
\noindent
Felipe Hernández\\
\textsc{{\'E}cole Polytechnique F{\'e}d{\'e}rale de Lausanne} (EPFL)\par\nopagebreak
\noindent
\href{mailto:felipe.hernandezcastro@epfl.ch}
{\texttt{felipe.hernandezcastro@epfl.ch}}

\bigskip
\noindent
Ioannis Kousek\\
\textsc{University of Warwick} \par\nopagebreak
\noindent
\href{mailto: ioannis.kousek@warwick.ac.uk}
{\texttt{ioannis.kousek@warwick.ac.uk}}

\bigskip
\noindent
Tristán Radić\\
\textsc{Northwestern University} \par\nopagebreak
\noindent
\href{mailto:tristan.radic@u.northwestern.edu}
{\texttt{tristan.radic@u.northwestern.edu}}

\end{document}